\documentclass[11pt,a4paper]{article}

\usepackage{fullpage}
\usepackage[latin1]{inputenc}
\usepackage{amsmath}
\usepackage{amsthm} 
\usepackage{amsfonts}
\usepackage{amssymb}
\usepackage{tikz}
\usepackage{enumitem} 
\usepackage{comment}

\usetikzlibrary{decorations.markings}
\usetikzlibrary{decorations.pathreplacing}

\newtheorem{theorem}{Theorem}[section]
\newtheorem{lemma}[theorem]{Lemma}
\newtheorem{proposition}[theorem]{Proposition}

\newtheorem{claim}[theorem]{Claim}
\theoremstyle{definition}
\newtheorem{definition}[theorem]{Definition}

\newcommand{\E}{\ensuremath{\mathbb{E}}}
\renewcommand{\Pr}{\ensuremath{\mathbb{P}}}
\newcommand{\sm}{\setminus}
\newcommand{\eps}{\varepsilon}
\newcommand{\paths}{\mathcal{P}}
\newcommand{\pbeg}{\mathrm{beg}}
\newcommand{\pend}{\mathrm{end}}
\newcommand{\pint}{\mathrm{int}}
\newcommand{\reg}{\mathrm{reg}}
\newcommand{\rand}{\mathrm{rand}}
\newcommand{\hnpk}{\ensuremath{H^{(k)}_{n, p}}}
\newcommand{\hnqk}{\ensuremath{H^{(k)}_{n, q}}}
\newcommand{\kFrac}{\ensuremath{\left\lceil\frac{k}{k-\ell}\right\rceil}}

\begin{document}
\author{Andrew McDowell\thanks{{\tt andrew.mcdowell@kcl.ac.uk}. Informatics Department, King's College London, London WC2R 2LS,
United Kingdom.}~~and Richard Mycroft\thanks{{\tt r.mycroft@bham.ac.uk}. School of Mathematics, University of Birmingham, Birmingham B15 2TT, United Kingdom. Research supported by EPSRC grant EP/M011771/1.}}
\title{Hamilton $\ell$-cycles in randomly perturbed hypergraphs}
\date{}
\maketitle

\begin{abstract}
We prove that for integers $2 \leq \ell < k$ and a small constant $c$, if a $k$-uniform hypergraph with linear minimum codegree is randomly `perturbed' by changing non-edges to edges independently at random with probability $p \geq O(n^{-(k-\ell)-c})$, then with high probability the resulting $k$-uniform hypergraph contains a Hamilton $\ell$-cycle. This complements a recent analogous result for Hamilton $1$-cycles due to Krivelevich, Kwan and Sudakov, and a comparable theorem in the graph case due to Bohman, Frieze and Martin.
\end{abstract}

\bibliographystyle{amsplain}
\section{Introduction}
Hamilton cycles are one of the most fundamental and widely studied structures in graph theory. We call a graph Hamiltonian if it contains a Hamilton cycle, that is, a cycle that covers all of the vertices of the graph. Many properties of Hamilton cycles in graphs are well understood, for example, minimum degree conditions~\cite{D} and random thresholds~\cite{P} 
that guarantee the existence of a Hamilton cycle.

A \emph{$k$-uniform hypergraph}, or \emph{$k$-graph}, is comprised of a vertex set and an edge set, where each edge consists of $k$ vertices. This generalises the notion of a graph (the case $k=2$). For $k$-graphs there a number of distinct but equally natural extensions of Hamiltonicity. Indeed, for $1 \leq \ell \leq k-1$ we say that a $k$-graph is an \emph{$\ell$-cycle} if there exists a cyclic ordering of the vertices of the graph such that every edge consists of $k$ consecutive vertices and each edge intersects the subsequent edge (in the natural order of the edges) in exactly $\ell$ vertices. We say that a $k$-graph~$H$ contains a Hamilton $\ell$-cycle if it contains an $\ell$-cycle as a spanning subgraph. Note that a necessary condition for this is that $k-\ell$ divides $n$, since every edge of the cycle contains exactly~$k-\ell$ vertices which were not contained in the previous edge.

Given integers $n$ and $k$ and a probability $p$, we can form a random $k$-graph with vertex set~$[n]$ by including each $k$-tuple of vertices as an edge with probability $p$, independently of all other choices. We denote the resulting random $k$-graph by $\hnpk$. This is the most well-studied notion of random $k$-graph, and generalises the Erd\H{o}s-Renyi random graph $G_{n, p} = H_{n, p}^{(2)}$.

\subsection{Threshold probabilities for Hamilton $\ell$-cycles in random $k$-graphs}

One of the most natural questions to ask is to identify threshold probabilities for the existence of Hamilton cycles in $\hnpk$. In the graph case the following theorem established very precise bounds on the critical probability for this property. This was independently proved by Bollob\'as~\cite{B} and by Koml\'os and Szemer\'edi~\cite{KSz}.

\begin{theorem}[\cite{B, KSz}] \label{randomthresholdsgraphs}
For every function $\omega(n)$ for which $\omega(n) \to \infty$ as $n \to \infty$, if $p \geq \frac{\log n + \log \log n + \omega(n)}{n}$ then with high probability $G_{n, p}$ contains a Hamilton cycle, whilst if $p \leq \frac{\log n + \log \log n - \omega(n)}{n}$ then with high probability $G_{n, p}$ does not contain a Hamilton cycle.
\end{theorem}

More recently Dudek and Frieze~\cite{DF1, DF2} largely answered the analogous question for $k$-graphs through bounds established in a pair of papers, which are combined together in the following theorem (the case $k=3, \ell=1$ was previously addressed by Frieze~\cite{F}).

\begin{theorem}[\cite{DF1, DF2}] \label{randomthresholdshypergraphs}
For every $\alpha > 0$, every $1 \leq \ell \leq k-1$ and every function $\omega(n)$ for which $\omega(n) \to \infty$ as $n \to \infty$ there exists $c, C > 0$ such that if 
$$p \geq \begin{cases}
C n^{-(k-1)} \log n & \mbox{if }\ell=1, k=3,\\
\omega(n) n^{-(k-1)} \log n & \mbox{if }\ell=1, k \geq 4,\\
\omega(n) n^{-(k-2)} & \mbox{if } \ell=2, \\
C n^{-(k-\ell)} & \mbox{if }3\leq \ell \leq k-1, \end{cases}
$$
then with high probability $\hnpk$ contains a Hamilton $\ell$-cycle, whilst if 
$$p \leq \begin{cases}
c n^{-(k-1)} \log n & \mbox{if }\ell=1,\\
c n^{-(k-\ell)} & \mbox{if } \ell \geq 2, \end{cases}
$$
then with high probability $\hnpk$ does not contain a Hamilton $\ell$-cycle. 
\end{theorem}

In particular, the upper and lower bounds on the critical probability are separated by a~$\omega(n)$-factor in the case $\ell = 1$ for $k \geq 4$, and in the case $\ell = 2$ for $k \geq 3$. In all other cases the difference is a constant factor.

\subsection{Dirac-type conditions for Hamilton $\ell$-cycles in $k$-graphs}

Dirac's Theorem, a classical result of graph theory~\cite{D}, states that any graph $G$ on $n \geq 3$ vertices with minimum degree $\delta(G) \geq n/2$ admits a Hamilton cycle. A great deal of research in recent years has focussed on finding analogous results for hypergraphs, using the following notions of minimum degree. Given a $k$-graph $H$ and a set $T$ of vertices of $H$ we define $\deg_H(T)$, the \emph{degree} of $T$, to be the number of edges in $H$ which contain $T$ as a subset (we omit the subscript when $H$ is clear from the context). For an integer $1\leq t \leq k-1$ the \emph{minimum $t$-degree} of $H$, denoted $\delta_t(H)$, is then defined to be the minimum value of $\deg(T)$ taken over all sets $T \subseteq V(H)$ with $|T| = t$. In particular, the parameters $\delta_1(H)$ and $\delta_{k-1}(H)$ are referred to as the \emph{minimum vertex degree} and \emph{minimum codegree} of $G$ respectively.

The following theorem collects together together the results of a series of papers by numerous authors over several years; it establishes asymptotically for every $k$ and $\ell$ the best-possible minimum codegree condition which guarantees the existence of a Hamilton $\ell$-cycle in a $k$-graph.

\begin{theorem}[\cite{HS, KKMO, KMO, KO, RRSz1, RRSz2}] \label{codeg}
  For any $k \geq 3$, $1 \leq \ell < k$ and $\eta > 0$, there exists~$n_0$ such that if $n \geq
  n_0$ is divisible by $k-\ell$ and $H$ is a $k$-graph on $n$ vertices with $$\delta_{k-1}(H) \geq 
  \begin{cases}
    \left( \frac{1}{2} + \eta \right) n& \mbox{ if $k-\ell $ divides $k$,} \\
    \left(\frac{1}{\lceil 
    \frac{k}{k-\ell} \rceil(k-\ell)}+\eta\right) n & \mbox{otherwise,}
  \end{cases} 
  $$ 
  then $H$ contains a Hamilton $\ell$-cycle. 
\end{theorem}

By contrast, the exact value of this threshold (for large $n$) has only been found in a small number of cases, namely for $k=3, \ell=2$ by R\"odl, Ruci\'nski and Szemer\'edi~\cite{RRSz3}, for $k=4, \ell=2$ by Garbe and Mycroft~\cite{GM}, for $k = 3$ and $\ell = 1$ by Czygrinow and Molla~\cite{CM} and for any $k \geq 3$ and $\ell < k/2$ by Han and Zhao~\cite{HZ}. For other degree conditions, less still is known; indeed the only cases in which the minimum $t$-degree threshold for a Hamilton $\ell$-cycle is known even asymptotically are the cases $k \geq 3, \ell < k/2, t={k-2}$ (due to Bastos, Mota, Schacht, Schnitzer and Schulenburg~\cite{BMSSS} with previous results for the case $(k,\ell, t) = (3,1,1)$ due to Bu\ss, H\`an and Schacht~\cite{BHS} and Han and Zhao~\cite{HZ2}) and $(k,\ell, t) = (3,2,1)$ (due to Reiher, R\"odl, Ruci\'nski, Schacht and Szemer\'edi~\cite{RRRSSz}).

For a much more detailed exposition of the results briefly described in this subsection we refer the reader to the recent surveys by K\"uhn and Osthus~\cite{KO2}, R\"odl and Ruci\'nski~\cite{RR} and Zhao~\cite{Z}.

\subsection{Hamilton $\ell$-cycles in randomly perturbed $k$-graphs}

Comparing the results of the previous two subsections, we observe that the random $k$-graphs around the threshold probability for containing a Hamilton $\ell$-cycle typically have far fewer edges than those whose minimum degree is close to the minimum degree threshold to force such a cycle. This invites the question of how far a typical graph of lower degree is from being Hamiltonian, motivating the following definition: given a $k$-graph $H$, the \emph{$p$-perturbation of $H$} is the $k$-graph~$H^+_p$ on the same vertex set in which every edge of $H$ is an edge of $H^+_p$ and additionally each $k$-tuple of vertices which is not an edge of $H$ is an edge of $H^+_p$ with probability~$p$, independently of all other $k$-tuples. For graphs this setup was considered by Bohman, Frieze and Martin~\cite{BFM}, who showed that if $G$ has linear minimum degree then adding a linear number of random edges suffices to ensure that $G_p^+$ has a Hamilton cycle.

\begin{theorem}[\cite{BFM}] \label{perturbedcyclesgraphs}
For every $\alpha > 0$ there exists $\lambda > 0$ such that if $G$ is a graph on $n$ vertices with $\delta(G) \geq \alpha n$ and $p \geq \lambda/n$ then with high probability $G_p^+$ contains a Hamilton cycle.
\end{theorem}

More recently Krivelevich, Kwan and Sudakov~\cite{KKS} established a similar result for loose cycles in perturbed $k$-graphs of linear minimum codegree.

\begin{theorem}[\cite{KKS}] \label{perturbed1cycles}
For every $k \geq 3$ and every $\alpha > 0$ there exists $\lambda > 0$ such that if $H$ is a $k$-graph on $n$ vertices with $\delta_{k-1}(H) \geq \alpha n$ and $p \geq \lambda n^{1-k}$ then with high probability $H^+_p$ contains a Hamilton $1$-cycle .
\end{theorem}

These results can be viewed as demonstrating the fragility of graphs and $k$-graphs which do not contain a Hamilton $1$-cycle, as a relatively small perturbation of these graphs will create such a cycle with high probability. Alternatively, comparing these results to the random thresholds presented earlier suggests another interpretation in terms of how many random edges must be added to a $k$-graph to create a Hamilton 1-cycle. Indeed, Theorem~\ref{randomthresholdsgraphs} and the case $\ell=1$ of Theorem~\ref{randomthresholdshypergraphs} show that if we start with an empty $k$-graph we must add around $n \log n$ edges to achieve this, whereas Theorems~\ref{perturbedcyclesgraphs} and~\ref{perturbed1cycles} show that if we start with a $k$-graph of linear minimum codegree then we need only add $O(n)$ edges, i.e. we save a factor of $\log n$ compared to starting with an empty $k$-graph.

After their proof of Theorem~\ref{perturbed1cycles}, Krivelevich, Kwan and Sudakov highlighted two natural directions for further research. The first of these is to consider if statements analogous to Theorem~\ref{perturbed1cycles} hold for Hamilton $\ell$-cycles where $\ell \geq 2$. Secondly, Theorem~\ref{perturbed1cycles} pertains only to $k$-graphs of high minimum codegree, which is the strongest form of minimum degree condition for $k$-graphs, and it is natural to consider whether a weaker notion of minimum degree would suffice instead. 

Our main result in this paper is the following theorem, which gives conditions for the existence of a Hamilton $\ell$-cycle in a randomly perturbed $k$-graph for any $2 \leq \ell \leq k-1$. Together with Theorem~\ref{perturbed1cycles} this answers the question for all forms of Hamilton $\ell$-cycle in $k$-graphs of linear minimum codegree. Moreover, for $2 \leq \ell \leq k-2$ we actually only require a weaker form of minimum degree condition. 

\begin{theorem}[Main result]\label{main}
Fix integers $2 \leq \ell < k$ and define $\ell' := \max(\ell, k-\ell)$. For every $\alpha > 0$ there exists $c > 0$ such that if $H$ is a $k$-graph on $n$ vertices such that $\delta_{\ell'}(H) \geq \alpha n^{k-\ell'}$ and $k-\ell$ divides $n$, then for $p \geq n^{-(k-\ell)-c}$ the $k$-graph $H^+_p$ contains a Hamilton $\ell$-cycle with high probability.
\end{theorem}

In particular, this shows that the number of random edges which need to be added to a $k$-graph of linear minimum codegree to guarantee the existence of a Hamilton $\ell$-cycle is $O(n^{\ell-c})$ for some constant $c > 0$. Perhaps surprisingly, this gives a saving of a polynomial factor in comparison to the number of random edges which mush be added to the empty $k$-graph to achieve this, which is around $n^{\ell}$ by Theorem~\ref{randomthresholdshypergraphs}. In other words, the `effect' of starting with a dense hypergraph is much stronger for $\ell \geq 2$ compared to the case $\ell = 1$, where we saved only a factor of $\log n$.

Theorem~\ref{main} is best-possible in the sense that it would not hold if the minimum degree condition $\delta_{\ell'}(H) \geq \alpha n^{k-\ell'}$ were replaced by any condition of the form $\delta_{\ell'}(H) \geq f(n) n^{k-\ell'}$ with $f(n) = o(n)$; we prove this assertion in Lemma~\ref{optimal}.

\subsection{Definitions and notation}

For integers $1 \leq \ell < k$ we say that a $k$-graph $P$ is an \emph{$\ell$-path} if its vertices can be linearly ordered so that every edge of $P$ consists of $k$ consecutive vertices and each edge of $P$ intersects the subsequent edge (in the natural order of the edges) in precisely $\ell$ vertices; the \emph{length} of $P$ is the number of edges of $P$. In particular, an $\ell$-path $P$ of length $m$ has $b := m(k-\ell)+\ell$ vertices, and we say that the sequence $(v_1, \dots, v_b)$ is a \emph{vertex sequence} for $P$ if $V(P) = \{v_1, \dots, v_b\}$ and the edges of $P$ are precisely the sets $\{v_{r(k-\ell)+1}, \dots, v_{r(k-\ell)+k}\}$ for $0 \leq r \leq m-1$. Note that~$P$ is uniquely determined by its vertex sequence, but there may be several vertex sequences for the same $\ell$-path. We say that an $\ell$-path $P$ is a \emph{path segment} or \emph{subpath} of an $\ell$-path $Q$ or $\ell$-cycle~$C$ to mean that $P$ appears as a subgraph of $Q$ or $C$.  

Given an $\ell$-path $P$ with vertex sequence $(v_1, \dots, v_b)$ we refer to the ordered $\ell$-tuples $P^{\pbeg} = (v_1, \dots, v_\ell)$ and $P^{\pend} = (v_{b-\ell+1}, \dots, v_b)$ as \emph{ends} of $P$. Since an $\ell$-path $P$ may have several vertex sequences, there may be multiple choices for ends of $P$ if a vertex sequence is not specified, but unless stated otherwise we make an arbitrary choice and simply refer to these as the ends of~$P$. The \emph{interior vertices} of $P$ are then the vertices which do not lie in either end of $P$, and we write $P^{\pint} := V(P) \sm (P^{\pbeg} \cup P^{\pend})$ for the set of interior vertices of $P$. Note that if $P$ and $Q$ are $\ell$-paths with $P^{\pend} = Q^{\pbeg}$ which have no vertices in common outside this set, then the $k$-graph $PQ$ with vertex set $V(P) \cup V(Q)$ and edge set $E(P) \cup E(Q)$ is an $\ell$-path also, with ends $P^{\pbeg}$ and $Q^{\pend}$. We will construct a Hamilton $\ell$-cycle in $H_p^+$ by connecting several $\ell$-paths in $H_p^+$ in this manner.

We say that a $k$-graph $H$ is \emph{$k$-partite} if there exists a partition of $V(H)$ into \emph{vertex classes} $V_1, \dots, V_k$ such that every edge $e \in H$ has $|e \cap V_i| = 1$ for each $i \in [k]$. Given a $k$-graph $H$ we write $e(H)$ for the number of edges of $H$, and $v(H)$ for the number of vertices of $H$. We also frequently identify a $k$-graph with its edge set, for example, writing $|H|$ for $e(H)$ and $e \in H$ to mean $e \in E(H)$. Given sets $S, T \subseteq V(H)$ we write $\deg_H(S, T)$ for the number of edges $e \in E(H)$ with $S \subseteq e$ and $e \sm S \subseteq T$. In other words $\deg_H(S, T)$ counts the number of ways to extend $S$ to an edge of $H$ by adding vertices from $T$. We omit the subscript when $H$ is clear from the context.

Given a function $\pi: U \to V$ and an ordered $k$-tuple $R = (u_1, \dots, u_k)$ of elements of $U$ we write $\pi(R)$ to denote the ordered $k$-tuple $(\pi(u_1), \dots, \pi(u_k))$. On the other hand, for an unordered subset $S \subseteq U$ we write $\pi(S)$ for the image of $S$ under $\pi$ in the usual manner. For an ordered $k$-tuple $R = (u_1, \dots, u_k)$ and an unordered set $S = \{v_1, \dots, v_k\}$ of size $k$ we sometimes abuse notation by writing $R = S$ to mean that $S$ and $u$ have precisely the same elements, that is, $\{u_1, \dots, u_k\} = \{v_1, \dots, v_k\}$.

Given a set $S$ and an integer $k$ we write $\binom{S}{k}$ for the set of subsets of $S$ of size $k$. We write $x \ll y$ to mean that for any $y > 0$ there exists $x_0 > 0$ such that for any $0 < x < x_0$ the following statement holds. Similar statements with more variables are defined similarly. Note carefully that this is not the same as the common usage of $\ll$ in probabilistic arguments in which we say $x \ll y$ if $x/y \to 0$; the latter usage of the symbol does not appear anywhere in this paper.

\subsection{Proof outline and structure of the paper}
The proof of Theorem~\ref{main} proceeds by an `absorbing' argument. This is a powerful technique for embedding large subgraphs in dense or random graphs and hypergraphs which has yielded many successes over the past two decades. To find a Hamilton $\ell$-cycle in a $k$-graph $H$, a typical absorbing argument consists of a `path cover lemma', an `absorbing lemma' and a `connecting lemma'. We follow the same top-level approach, but each of these three components must be tailored to the perturbing setting, as described below.

\medskip\noindent {\bf Path cover lemma.} We use a special case of a seminal theorem of Johanssen, Kahn and Vu regarding perfect tilings in $k$-graphs. This special case states that under the conditions of Theorem~\ref{main}, we can find a spanning collection $\paths$ of vertex-disjoint $\ell$-paths in $H_p^+$ of length close to $\frac{\ell-1}{c}$. In fact, we find these paths entirely in $\hnpk$, and do not appeal to the minimum degree condition of $H$ at all. The reason for this is shown by the construction we present in Lemma~\ref{optimal} to show that Theorem~\ref{main} is in a sense optimal; this construction demonstrates that some $k$-graphs $H$ satisfying $\delta_{k-1}(H) \geq \alpha n$ can provide only a few edges towards $\paths$. We note that since each path in $\paths$ has constant length, the size of $\paths$ is linear in $n$. This differs from typical previous applications of the absorbing method, in which we choose a constant number of paths of linear length.

\medskip \noindent {\bf Absorbing lemma.}
Our absorbing lemma states that, under the conditions of Theorem~\ref{main}, we can find a collection $\paths$ of vertex-disjoint $\ell$-paths of constant length in $H$ so that for almost all sets $S \in \binom{V(H)}{k-\ell}$ there are many paths $P \in \paths$ which can `absorb' $S$ in $H_p^+$. By this we mean that there is an $\ell$-path $Q$ in $H_p^+$ with the same ends as $P$ whose vertices are the vertices in $P$ and those of $S$. The point of this is that if $P$ is a path segment of an $\ell$-cycle $C$ which does not include any vertex of $S$, then we may replace $P$ by $Q$ in $C$ and thereby `absorb' the vertices of $S$ into $C$. 
To prove this lemma we first present an `absorbing structure' $F$, which contains an $\ell$-path $P$ and also a set $F_A$ of size $k-\ell$ which can be absorbed into $P$ in $F$. Additionally the edges of $F$ are partitioned into a `regular' part $F_{\reg}$ and a `random' part $F_{\rand}$. We show that almost all ordered $(k-\ell)$-tuples of vertices of $H$ extend to many copies of $F_{\reg}$ in $H$ (i.e. not using any random edges); in fact the number of extensions is a constant proportion of the maximum possible number of extensions. We then show that when we expose the random edges of $H_p^+$, with high probability many of these extensions $F_{\reg}$ gain the required edges of $F_{\rand}$ to form a copy of $F$. Together this shows that almost all ordered $(k-\ell)$-tuples of vertices of~$H$ extend to many copies of $F$ in $H_p^+$. We then randomly select a linear-size set of copies of~$F \sm F_A$, and show that almost all ordered $(k-\ell)$-tuples are extended to a copy of $F$ by many of these copies, so taking the paths $P$ from each of these copies then gives the desired collection of absorbing paths (after removing the small number of paths which intersect and adding a few extra paths to cover a small number of atypical vertices).

\medskip \noindent {\bf Connecting lemma.}
Our `connecting lemma' states that, under the conditions of Theorem~\ref{main}, given a collection $\paths$ of $\ell$-paths in $H$ we can find an $\ell$-cycle $C$ in $H^+_p$ which includes every path $P \in \paths$ as a path segment. To illustrate the proof of this, let $P$ and $Q$ be $\ell$-paths which we wish to connect. Then we use the minimum degree condition to show that there are many possible ways to extend $P$ and $Q$ each by $t := \kFrac-1$ edges in $H$ without overlapping. Indeed, the number of ways to do this is sufficient that, when we expose our random edges, with high probability some of these extensions are joined by $t$ random edges to form an $\ell$-path of length~$3t$ in $H_p^+$ which connects $P$ and $Q$ into a single long $\ell$-path. In fact we show that with high probability we can do this while avoiding any given small set of vertices, which allows us to iterate the connections to connect all the paths in $\paths$ into a cycle. We note that extending each of $P$ and $Q$ by $t$ edges from $H$ ensures that the $t$ random edges we use to complete the connection do not intersect the original vertices of $P$ or $Q$; this fact is crucial for us to have sufficiently many connecting paths in the random graph.

\medskip \noindent {\bf Proof of Theorem~\ref{main}.}
Finally, to prove Theorem~\ref{main} we combine the aforementioned lemmas in the following way. We form $H_p^+$ by exposing edges in four rounds, permitting two applications of the connecting lemma and one application each of the absorbing lemma and path cover lemma. First, we apply the absorbing lemma to obtain a collection $\paths$ of `absorbing' $\ell$-paths so that almost all sets $S \in \binom{V(H)}{k-\ell}$ are `good' in the sense that there are many paths~$P \in \paths$ which can `absorb' $S$ in $H_p^+$. We then apply the connecting lemma to find a single $\ell$-path $P$ which contains each path in $\paths$ as a path segment. Following this we randomly select a small reservoir set $R$, before applying the path cover lemma to find vertex-disjoint $\ell$-paths of long constant length which cover every vertex of $H$ except for those in $V(P)$ or $R$. We then make a second application of the connecting lemma to find an $\ell$-cycle $C$ in $H^+_p$ which includes $P$ and each of these $\ell$-paths as a path segment. The cycle $C$ then covers every vertex of $H$ except for those in the reservoir $R$ which were not used for the second application of the connecting lemma. Finally, we complete the proof by partitioning these leftover vertices into good $(k-\ell)$-tuples and greedily absorbing these into the `absorbing' paths obtained from the absorbing lemma (which are now path segments of $C$). We note that it is necessary to make two separate applications of the connecting lemma here since our collection $\paths$ of `absorbing paths' is too large to be connected using the reservoir set~$R$ (which in turn cannot be any larger or we would be unable to absorb all the leftover vertices into $C$).

\medskip \noindent {\bf Structure of the paper.}
In Section~\ref{sec:mainproof} of this paper we give formal statements of the three principal lemmas described above, but we defer the proofs of the absorbing and connecting lemmas to subsequent sections. Also in Section~\ref{sec:mainproof} we present the full proof of Theorem~\ref{main} as outlined above, and give a construction which demonstrates the optimality of Theorem~\ref{main}. Following this, in Section~\ref{sec:absorbing} we prove our `absorbing lemma', and in Section~\ref{sec:connecting} we turn to the proof of our `connecting lemma'. Finally, we make some concluding remarks in Section~\ref{sec:conclusion}.

\section{Proof of Theorem~\ref{main} and its optimality} \label{sec:mainproof}

The first subsection of this section includes the statements of the key lemmas described in the proof outline, whilst in the second we use these to prove Theorem~\ref{main}. The final subsection gives examples demonstrating the optimality of Theorem~\ref{main}.

\subsection{Key lemmas}
As described above, our `path cover lemma' is provided by a special case of a seminal theorem of Johanssen, Kahn and Vu~\cite{JKV} regarding the threshold probability for the existence of an $H$-factor in $\hnpk$ (an \emph{$H$-factor} in a $k$-graph $G$ is a spanning collection of vertex-disjoint copies of~$H$ in $G$). For any $k$-graph $H$ define $d(H) := e(H)/(v(H)-1)$, and say that $H$ is \emph{strictly balanced} if $d(H') < d(H)$ for every proper subgraph $H' \subsetneq H$. Johanssen, Kahn and Vu showed that if a $k$-graph $H$ is strictly balanced, then $th_H(n)$ is a probability threshold for the existence of an $H$-factor in $\hnpk$, where
$$th_H(n)=n^{-1/d(H)} \left( \log n\right)^{1/e(H)}.$$
Observe that if $P$ is an $\ell$-path $k$-graph of length $m$ then, since $P$ has $(k-\ell)m + \ell$ vertices, we have
$$d(P) = \frac{m}{(k-\ell)m+\ell -1}=\frac{1}{(k-\ell)+\frac{\ell -1}{m}}.$$
So for $2 \leq \ell < k$ we find that $d(P)$ increases as $m$ increases, and it follows that $P$ is strictly balanced. We therefore have the following theorem (the special case of Johanssen, Kahn and Vu's theorem for $\ell$-path $k$-graphs). 

\begin{theorem} \label{pathtiling}
Fix integers $2 \leq \ell < k$ and $m \geq 1$, and define $b := (k-\ell)m + \ell$. Let $P$ be the $\ell$-path $k$-graph of length~$m$, so $P$ has $b$ vertices. If $b$ divides $n$ and $p = \omega \left(\log n^{1/m} n^{-(k-\ell) -(\ell-1)/m}\right)$ then with high probability $\hnpk$ contains a $P$-factor.
\end{theorem}

Our `absorbing lemma' is the next lemma, and is proved in Section~\ref{sec:absorbing}. For this we make the following definition: given a $k$-graph $H$, an $\ell$-path $P$ in $H$ and a $(k-\ell)$-tuple $S \in \binom{V(H)}{k-\ell}$, we say that $P$ can \emph{absorb} $S$ in $H$ if there exists an $\ell$-path $Q$ in $H$ with the same ends as $P$ and vertex set $V(Q) = V(P) \cup S$.

\begin{lemma}\label{Lem:Absorbing_Paths}
Fix integers $2 \leq \ell < k$, define $t := \kFrac - 1$ and fix a constant $c < 1/t$. Suppose that $\xi \ll \eta \ll \alpha, 1/k$, and let
$H$ be a $k$-graph on $n$ vertices with $\delta_{k-\ell}(H) \geq\alpha n^{\ell}$. If $p \geq n^{-(k-\ell)-c}$, then with high probability there exists a collection $\paths$ of at most $\eta n$ vertex-disjoint $\ell$-paths in $H$  and a set $\mathcal{B} \subseteq \binom{V(H)}{k-\ell}$ such that
\begin{enumerate}[label=(\alph*), noitemsep]
\item each path in $\paths$ has at most $3k^2$ vertices,
\item each vertex in $V(H) \sm \bigcup_{P \in \paths} V(P)$ lies in at most $\eta n^{k-\ell-1}$ elements of $\mathcal{B}$, and
\item for each $(k-\ell)$-tuple $S \notin \mathcal{B}$ there are at least $\xi n$ paths $P \in \paths$ which can absorb $S$ in $H^+_p$.
\end{enumerate}
\end{lemma}

Our `connecting lemma' is the following lemma, allowing us to connect a collection of paths to form a single cycle. We prove this lemma in Section~\ref{sec:connecting}.

\begin{lemma}\label{Lemma:ConnectingPaths}
Fix integers $2 \leq \ell < k$, define $t := \kFrac - 1$ and fix a constant $c < 1/t$. Suppose that $\vartheta \ll \alpha, 1/k$, let $H$ be a $k$-graph on~$n$ vertices and let $\paths$ be a collection of at most $\vartheta n$ vertex-disjoint $\ell$-paths in $H$. Suppose that, writing $X := V(H) \sm \bigcup_{P \in \paths} V(P)$, for every set $S \in \binom{V(H)}{\ell}$ we have $\deg_H(S, X) \geq \alpha n^{k-\ell}$. If $p \geq n^{-(k-\ell)-c}$ then with high probability there exists an $\ell$-cycle~$C$ in $H_p^+$ with $|V(C) \cap X| \leq 4k |\paths|$ such that $C$ contains each $P \in \paths$ as a path segment.
\end{lemma}

Sometimes we will apply Lemma~\ref{Lemma:ConnectingPaths} to obtain an $\ell$-path which contains each $P \in \paths$ as a path segment (rather than an $\ell$-cycle with this property). This can be achieved by simply deleting edges from the cycle given by Lemma~\ref{Lemma:ConnectingPaths}, so we will do so without further comment.

Finally, we also use the following theorem of Daykin and H\"aggkvist \cite{DH}. This states that every $k$-graph with sufficiently high minimum vertex degree admits a perfect matching, i.e. a spanning collection of disjoint edges.

\begin{theorem}\label{Thm:MinDegreeMatching} If $k \geq 2$ and $k$ divides $n$, then every $k$-graph $H$ of order $n$ with $\delta_1(H) > \frac{k-1}{k}\left(\binom{n-1}{k-1}-1\right)$ contains a perfect matching.
\end{theorem}

\subsection{Proof of Theorem~\ref{main}}

We now combine these lemmas to prove Theorem~\ref{main}.

\begin{proof}[Proof of Theorem~\ref{main}]
Fix $2 \leq \ell < k$ and write $\ell' := \max(\ell, k-\ell)$. Given $\alpha > 0$, introduce new constants satisfying 
$$1/m \ll \xi \ll \eta \ll \alpha, 1/k,$$
and define $\alpha' := \alpha/k!$ and $b := m(k-\ell) + \ell$.
Let $H$ be a $k$-graph on~$n$ vertices, where $k-\ell$ divides $n$, and suppose that $\delta_{\ell'}(H) \geq \alpha n^{k-\ell'}$, from which it follows that $\delta_{\ell}(H) \geq \alpha' n^{k-\ell}$ and $\delta_{k-\ell}(H) \geq \alpha' n^{\ell}$. Finally, fix $c < (\ell-1)/m$ and $p \geq n^{-(k-\ell)-c}$. We proceed by a multiple exposure argument with four rounds. For this, let $H_1$, $H_2$, $H_3$ and $H_4$ be independently drawn from $H^{(k)}_{n, p/4}$. Then by a standard coupling argument we may assume that $H \cup \bigcup_{i \in [4]} H_i \subseteq H_p^+$.

We begin by using our first exposure round to apply our absorbing lemma, Lemma~\ref{Lem:Absorbing_Paths}. This states that with high probability there exists a collection $\paths$ of at most $\eta n$ vertex-disjoint $\ell$-paths in $H \cup H_1$ and a set $\mathcal{B} \subseteq \binom{V(H)}{k-\ell}$ of $(k-\ell)$-tuples such that, writing $V(\paths) := \bigcup_{P \in \paths} V(P)$,
\begin{enumerate}[label=(\alph*), noitemsep]
\item each path in $\paths$ has at most $3k^2$ vertices,
\item each vertex in $V(H) \sm V(\paths)$ lies in at most $\eta n^{k-\ell-1}$ elements of $\mathcal{B}$, and
\item for each $(k-\ell)$-tuple $S \notin \mathcal{B}$ there are at least $\xi n$ paths $P \in \paths$ which can absorb $S$ in $H \cup H_1$.
\end{enumerate}

Let $X := V \sm V(\paths)$. Note that for every set $S \in \binom{V(H)}{\ell}$ there are $\deg_H(S) \geq \delta_{\ell}(H) \geq \alpha' n^{k-\ell}$-many $(k-\ell)$-tuples $S'$ for which $S \cup S'$ is an edge of $H$. Since $|\bigcup_{P \in \paths} V(P)| \leq 3k^2\eta n \leq \alpha' n/2$, at most $\alpha' n^{k-\ell}/2$ such sets $S'$ intersect $V(\paths)$, and so we have $\deg_H(S, X) \geq \alpha' n^{k-\ell}/2$. We now make our second exposure round to apply our connecting lemma, Lemma~\ref{Lemma:ConnectingPaths}, with $\eta$ and $\alpha'/2$ in place of $\vartheta$ and $\alpha$ respectively. With high probability this yields a single $\ell$-path~$P$ in $H \cup H_1 \cup H_2$ such that $P$ contains each path in $\paths$ as a path segment and $|V(P)| \leq |V(\paths)| + 4k|\paths| \leq 3k^2 \eta n + 4k\eta n \leq \alpha' n/2$. 

Define $X':= V \sm V(P)$. By exactly the same argument as above it follows that every set $S \in \binom{V(H)}{\ell}$ has $\deg_H(S, X') \geq \alpha' n^{k-\ell}/2$. 
Choose $r$ with $\xi n \leq r \leq 2\xi n$ such that $b$ divides~$|X'| - r$, and choose a `reservoir' set $R$ of size $r$ uniformly at random from all subsets of $X'$ of this size. Then for each set $S \in \binom{V(H)}{\ell}$, a standard Chernoff-type bound yields that
\begin{equation*}
\Pr\left(\deg_H(S, R)\leq \frac{\alpha' r^{k-\ell}}{4}\right)\leq \exp({-\Omega (n^{k-\ell})}).
\end{equation*}
Let $\mathcal{B}[R] \subseteq \mathcal{B}$ consist of all members of $\mathcal{B}$ which are subsets of $R$. Then using (b) we find that for every vertex $v \in V(H) \sm V(\paths)$ the expected number of sets in $\mathcal{B}[R]$ containing $v$ is at most~$2\eta r^{k-\ell-1}$. So by a standard Chernoff-type bound, for each $v \in V(H) \sm V(\paths)$ the probability that $v$ is contained in more than $3 \eta r^{k-\ell-1}$ members of $\mathcal{B}[R]$ is at most $\exp({-\Omega (n^{k-\ell})})$. So, taking a union bound over all such sets $S \in \binom{V(H)}{\ell}$ and vertices $v \in V(H) \sm V(\paths)$ we find that with high probability every set $S \in \binom{V(H)}{\ell}$ satisfies
\begin{equation}\label{eqn:minDegreeinR}\deg_H(S, R)> \frac{\alpha' r^{k-\ell}}{4} \geq \frac{\xi^{k-\ell}\alpha'}{4}n^{k-\ell}. 
\end{equation}
and 
\begin{equation}\label{vdeginR} \mbox{every vertex $v \in V(H) \sm V(\paths)$ is contained in at most $3 \eta r^{k-\ell-1}$ members of $\mathcal{B}[R]$.}
\end{equation}

We now make make our third exposure round to apply our path cover lemma, Theorem~\ref{pathtiling}. Note for this that $b$ divides $|X' \sm R|$ by our choice of $r$. So with high probability we obtain a spanning collection $\paths'$ of vertex-disjoint $\ell$-paths of length $m$ in $H_3[X' \sm R]$. Note in particular that $|\paths'| < n/b$. Let $\paths'' := \paths \cup \{P\}$, so $\paths''$ is a collection of at most $1+n/b \leq n/m$ vertex-disjoint $\ell$-paths in $H \cup H_1 \cup H_2 \cup H_3$, none of which intersect $R$.

Our fourth and final exposure round is to apply our connecting lemma, Lemma~\ref{Lemma:ConnectingPaths}, which we do with~$\xi^{k-\ell} \alpha'/4$ and $1/m$ here playing the roles of $\alpha$ and $\vartheta$ respectively there (so~\eqref{eqn:minDegreeinR} ensures that the degree condition of the theorem is satisfied). With high probability this yields an $\ell$-cycle~$C$ in $H_p^+$ with $|V(C) \cap R| \leq 4kn/m$ which contains each path in $\paths''$ as a path segment (so in particular~$C$ contains each path in $\paths$ as a path segment).

Finally we use the absorbing paths within $C$ to absorb the remaining vertices. Let $R' := R \sm V(C)$, so $R'$ consists of all vertices of $R$ except those used to connect paths in the fourth exposure round. Since $C$ contains every vertex outside $R$ we also have $R' = V(H) \sm V(C)$. By assumption $k-\ell$ divides $n$ and since $C$ is an $\ell$-cycle $k-\ell$ divides $|V(C)|$. It follows that $k-\ell$ divides $n-|V(C)| = |R'|$. Define an auxiliary $(k-\ell)$-graph $G$ with vertex set~$R'$ and edge set $\binom{R'}{k-\ell} \sm \mathcal{B}$. Then for every $v \in R'$, since $v \notin V(\paths)$, by~\eqref{vdeginR} and the fact that $|R'| \geq |R| - |V(C) \cap R| \geq r - 4kn/m \geq (1-\eta)r$ we have
$$\deg_G(v) \geq \binom{|R'|}{k-\ell-1} - 3\eta r^{k-\ell-1} > \frac{k-\ell-1}{k}\binom{|R'|}{k-\ell-1}$$
We may therefore apply Theorem~\ref{Thm:MinDegreeMatching} to find a perfect matching $M$ in $G$. 
We then have $|M| = |R'|/k \leq r/k \leq \xi n$, so by (c) we can greedily assign each $(k-\ell)$-tuple $e \in M$ to a distinct path in $\paths$ which can absorb it in $H \cup H_1$. Since each of these paths is a path segment of $C$, by absorbing each $e \in M$ in this way we obtain a Hamilton $(k-\ell)$-cycle in $H_p^+$.
\end{proof}

\subsection{Optimality of Theorem~\ref{main}}

We now demonstrate the optimality of Theorem~\ref{main} in terms of the minimum degree condition and the bound on $p$. To do this we use the following simple application of the first moment method, which gives gives conditions under which $H_p^+$ does not contain a tiling of $\ell$-paths of length $m$ which covers at least half of the vertices of $H$.

\begin{lemma} \label{firstmoment}
Fix integers $2 \leq \ell < k$ and $m$ and a constant $c > \ell/m$. Let $P$ be the $\ell$-path $k$-graph of length $m$, so $P$ has $b:= m(k-\ell)+\ell$ vertices. For $p < n^{-(k-\ell)-c}$, with high probability there does not exist a set of at least $n/2b$ vertex-disjoint $\ell$-paths each of length $m$ in $\hnpk$.
\end{lemma}

\begin{proof}
For simplicity we assume that $2b$ divides $n$; assuming otherwise makes very little difference to the calculations but is notationally more awkward.
Let $\paths$ be the $k$-graph formed by the disjoint union of $n/2b$ copies of $P$, and say that an injective function $f : V(\paths) \to V(\hnpk)$ is \emph{path-inducing} if $f(e) \in \hnpk$ for every $e \in \paths$. Let $X$ be the random variable which counts the number of path-inducing injective functions $f : V(\paths) \to V(\hnpk)$, and note that if there exists a set of $n/2b$ vertex-disjoint $\ell$-paths in $\hnpk$ then $X \geq 1$. Then since $\paths$ has $mn/2b$ edges and~$n/2$ vertices we have 
$$\E(X) = \binom{n}{n/2} (n/2)!p^{mn/2b} \leq 2^n \sqrt{2\pi n}\left(\frac{n}{2e}\right)^{n/2} (p^{m/b})^{n/2} < \sqrt{2\pi n} \cdot \left(n^{1-(k-\ell+c)m/b}\right)^{n/2}$$
By our choice of $c$ we have 
$(k-\ell+c)m/b = \frac{k-\ell+c}{k-\ell + \ell/m} > 1$, and it follows that $\E(X) = o(n)$, so with high probability we have $X= 0$ as required.
\end{proof}

\begin{lemma} \label{optimal}
Fix integers $2 \leq \ell < k$. For every $0 < \alpha < \frac{1}{12k^2}$ there exists a $k$-graph $H$ on $n$ vertices with $\delta_\ell (H)\geq \alpha \binom{n}{k-\ell}$ such that for every $c > \ell/\lfloor\frac{1}{4\alpha k}\rfloor$, if $p < n^{-(k-\ell) - c}$ then with high probability $H_p^+$ does not contain a Hamilton $\ell$-cycle.
\end{lemma}

\begin{proof}
Given $n$ and $\alpha$ we define a $k$-graph $ H := H(n, \alpha)$ as follows. First let $A$ and $B$ be disjoint sets with $|A| = \alpha n$ and $B = n-|A|$. Then take $V := A \cup B$ to be the vertex set of~$H$, so $v(H) = n$, and take every $e \in \binom{V}{k}$ with $V \cap A \neq \emptyset$ to be an edge of $H$. It follows that $\delta_{k-1}(H) \geq \alpha n$ so certainly we have $\delta_{\ell}(H) \geq \alpha\binom{n}{k-\ell}$ for each $\ell \in [k-1]$.

Suppose that $H_p^+$ contains a Hamilton $\ell$-cycle $C$, so $C$ has $n$ vertices and $n/(k-\ell)$ edges. Label the vertices of $C$ with $[n]$ in the natural order. Now fix $m := \lfloor \frac{1}{4\alpha k}\rfloor$ and note that by our assumption on $\alpha$ we then have $3k \leq m \leq \frac{1}{4\alpha k}$. Write $b:= m(k-\ell)+\ell$ and $r := \lfloor n/(k-\ell)(m+k)\rfloor$. For each $i \in [r]$ let $P_i$ be the subpath of $C$ of length $m$ starting at vertex $(i-1)(k-\ell)(m+k)+1$. Since each $P_i$ has $(k-\ell)m+\ell < (k-\ell)(m+k)$ vertices the paths $P_1, \dots, P_r$ are vertex-disjoint subpaths of~$C$. At most $|A| \leq \alpha n$ of the paths $P_i$ contain a vertex of $A$, so removing these we obtain a collection $\paths$ of vertex-disjoint $\ell$-paths of length~$m$ in $H_p^+[B]$ of size
$$ |\paths| \geq r - \alpha n \geq \frac{n}{(k-\ell)(m+k)} - 1 - \alpha n \geq \left(\frac{3}{4} - km\alpha\right) \frac{n}{(k-\ell)m} \geq  \frac{1}{2} \cdot \frac{n}{m(k-\ell)} \geq \frac{n}{2b},$$
where the third and fourth inequalities hold by our assumptions on the size of $m$.
However, since~$H[B]$ is empty it follows that we have a collection of $n/2b$ vertex-disjoint $\ell$-paths of length~$m$ in $\hnpk$. For $c > \ell/m$ and $p < n^{-(k-\ell)+c}$ this event has probability $o(n)$ by Lemma~\ref{firstmoment}, so we conclude that with high probability there is no Hamilton $\ell$-cycle in $H_p^+$.
\end{proof}

It follows from Lemma~\ref{optimal} that for any constant $c > 0$ Theorem~\ref{main} would not remain valid if the minimum degree condition were weakened to a condition that $\delta_{\ell'}(H) \geq f(n) n^{k-\ell}$ for some $f(n) = o(n)$. Likewise, for small $\alpha > 0$ the constant $c$ of Theorem~\ref{main} must satisfy $c < \ell/\lfloor\frac{1}{4\alpha k}\rfloor$, that is, $c$ declines with $\alpha$.

\section{The absorbing lemma} \label{sec:absorbing}

In this section we prove our absorbing lemma, Lemma~\ref{Lem:Absorbing_Paths}. Our proof makes use of the following Chernoff-type bounds for sums of independent random selections of bounded integers. We omit the proofs of these since they are essentially identical to the proof of Chernoff's bound for binomially distributed random variables (see e.g.~\cite{JLR}), which is the case when $m = x_1 = \dots = x_t = 1$. We also use the first of these in the proof of our connecting lemma in Section~\ref{sec:connecting}.

\begin{proposition}\label{chernoffbound}
Fix $p \in [0, 1]$ and integers $x_1, \dots, x_t$ with $0 \leq x_i \leq m$ for each $i \in [t]$, and write $S := \sum_{i \in [t]} x_i$. Randomly form a subset $I \subseteq [t]$ by including each $i \in [t]$ in $I$ with probability~$p$ and independently of all other choices, and let $X := \sum_{i \in I} x_i$. Then $\E(X) = pS$, and for $0 < \delta < 3/2$ we have
$$\Pr\left( X \leq (1-\delta) \E(X)\right) \leq \exp\left(-\frac{\delta^2\E(X)}{2m}\right) \mbox{ and } \Pr\left( X \geq (1+\delta) \E(X)\right) \leq \exp\left(-\frac{\delta^2\E(X)}{3m}\right).$$
\end{proposition}

\begin{proposition}\label{binomchernoffbound}
Fix $p \in [0, 1]$ and let $X_1, \dots, X_n$ be independent random variables with $\Pr(X_i = m) = p$ and $\Pr(X_i = 0) = 1-p$ for each $i \in [n]$. Define $X := \sum_{i \in [n]} X_i$. Then $\E(X) = pnm$, and for $t > 6\E(X)$ we have
$$\Pr\left( X \geq \E(X) + t \right) 
\leq \exp\left(-\frac{t}{m}\right).$$
\end{proposition}

We begin by defining the absorbing structures which we will use in the proof of Lemma~\ref{Lem:Absorbing_Paths}. 

\begin{definition} \label{defF}
For integers $2 \leq \ell < k$ we define $k$-graphs $F, F_{\reg}$ and $F_{\rand}$ as follows. Set $T := 3(k-\ell)+1$, and note that $Tk \equiv k \pmod {k-\ell}$, so there exist $\ell$-path $k$-graphs with $L := Tk$ vertices. The vertex set of $F$ is 
$$V(F) := \{v_i^j : i \in [k], j \in [T] \} \cup \{a_1, \dots, a_{k-\ell}\},$$
and $V(F)$ is partitioned into vertex classes $V_i$ for $i \in [k]$, where for each $i \in [k-\ell]$ we have $V_i := \{v_i^j : j \in [T]\} \cup \{a_i\}$ and for each $i \in [k] \sm [k-\ell]$ we have $V_i := \{v_i^j : j \in [T]\}$. Furthermore, the edges of $F$ are all sets which are edges of the $\ell$-path $P(F)$ with vertex sequence
\begin{equation}\label{firstpath}(v_1^1, v_2^1, \dots v_k^1, v_1^2, v_2^2,\dots, v_k^2, \dots, v_1^T, \dots, v_k^T)\end{equation}
and all sets which are edges of the $\ell$-path $Q(F)$ with vertex sequence
\begin{align}
\big(&v_1^1,\dots, v_k^1, \nonumber \\
&a_{1}, v_2^2,\dots, v_k^2, \nonumber\\
&v_1^3, a_2, v_3^3, \dots, v_k^3,\nonumber\\
&v_1^4, v_2^4, a_3, v_4^4\dots, v_k^4,\nonumber\\
& \vdots \nonumber\\
& v_1^{k-\ell+1}, \dots, v_{k-\ell-1}^{k-\ell +1}, a_{k-\ell}, v_{k-\ell+1}^{k-\ell+1}, \dots, v_{k}^{k-\ell+1},\nonumber\\
&v_{1}^{k-\ell+2}, \dots, v_{k}^{k-\ell+2}, \nonumber\\ 
& \vdots \nonumber \\
&v_{1}^{2(k-\ell)}, \dots, v_{k}^{2(k-\ell)}, \nonumber\\ 
&v_{1}^{2},v_{2}^{3},\dots, v_{k-\ell}^{k-\ell+1}, \nonumber\\
&v_{1}^{2(k-\ell)+1}, \dots, v_{k}^{2(k-\ell)+1}, \nonumber\\
& \vdots \nonumber \\ 
&v_{1}^{T}, \dots, v_{k}^{T}\big). \label{secondpath}
\end{align}
In other words, the vertex sequence for $Q(F)$ is formed by the following modifications of the vertex sequence for $P(F)$ in~\eqref{firstpath}: first replace $v_i^{i+1}$ by $a_i$ for each $i \in [k-\ell]$, then insert the replaced vertices as a consecutive subsequence immediately following $v_k^{2(k-\ell)}$.

Set $F_A := \{a_1, \dots, a_{k-\ell}\}$, and define $F_{\rand}$ to be the $k$-graph with vertex set $V(F) \sm F_A$ whose edges are all edges of $Q(F)$ which contain both $v_{k-\ell}^{k-\ell+1}$ and~$v_{1}^{2(k-\ell)+1}$. Also define $F_{\reg}$ to be the $k$-graph on vertex set $V(F)$ with edge set $F \sm F_{\rand}$. Finally set $F^{\pbeg} := (v_1^1,\dots v_\ell^1)$ and $F^{\pend} := (v_{k-\ell+1}^{T}, \dots, v_{k}^{T})$; we refer to $F^{\pbeg}$ and $F^{\pend}$ as the \emph{ends} of~$F$.
\end{definition}

The following properties of $F$ follow immediately from the definition.

\begin{proposition}
For every $k \geq 3$ the $k$-graph $F$ defined above satisfies the following properties.
\begin{enumerate}[noitemsep, label=(\roman*)]
\item $P(F)$ is an $\ell$-path in $F$ with vertex set $V(F) \sm F_A$ and ends $F^{\pbeg}$ and $F^{\pend}$.
\item $Q(F)$ is an $\ell$-path in $F$ with vertex set $V(F)$ and ends $F^{\pbeg}$ and $F^{\pend}$.
\item $F$ has $ L + k - \ell = Tk+k-\ell < 3k^2$ vertices.
\item $F_{\reg}$ is $k$-partite with vertex classes $V_1, \dots, V_k$.
\item $F_{\rand}$ has $L$ vertices and consists of an $\ell$-path of length $t := \kFrac - 1$ which has no vertices in common with $F_A$, $F^{\pbeg}$ or $F^{\pend}$, and also $L-t(k-\ell)-\ell$ isolated vertices.
\end{enumerate}
\end{proposition}

\begin{proof}
Both (i) and (ii) are immediate from our choice of edges of $F$ and $F(A)$, and~(iii) from our choice of~$V(F)$, since $Tk+k-\ell = (3k+1)(k-\ell)+k < 3k^2$. For~(iv) observe that every sequence of $k$ consecutive vertices in~\eqref{firstpath} or~\eqref{secondpath} contains one vertex from each of the $k$ vertex classes, except for the sequences of $k$ consecutive vertices in~\eqref{secondpath} which contain both $v_{k-\ell}^{k-\ell+1}$ and $v_{1}^{2(k-\ell)+1}$, but each edge of this type is in $F_{\rand}$ by definition. Finally, for~(v) observe that $v_{1}^{2(k-\ell)+1}$ is the $((2k+1)(k-\ell)+1)$-th vertex of the sequence in~\eqref{secondpath}, and therefore is the first vertex of some edge. Since each edge of an $\ell$-path contains $k-\ell$ vertices which are not contained in the subsequent edge the number of edges which contain both $v_{k-\ell}^{k-\ell+1}$ and $v_{1}^{2(k-\ell)+1}$ is $\left\lfloor \tfrac{k-1}{k-\ell} \right\rfloor = \kFrac - 1$, as claimed.
\end{proof}

In particular, properties (i) and (ii) show that the $\ell$-path $P(F)$ can absorb $A$ in $F$ (recall that this means there is an $\ell$-path $Q$ in $F$ with vertex set $P(F) \cup A$ and with the same ends as~$P(F)$, and $Q(F)$ has this property).

As described in the introduction, to prove Lemma~\ref{Lem:Absorbing_Paths} we first show that almost all ordered $(k-\ell)$-tuples $S$ of vertices of $H$ extend to many copies of $F_{\reg}$ in $H$ in which $S$ plays the role of~$F_A$ (no random edges are involved in this step). We do this in Lemma~\ref{extension}. Following this, we show that when we expose the random edges of $H_p^+$ we find with high probability that many of these extensions in fact give copies of $F$ in $H_p^+$; this is done in Lemma~\ref{numberofrandompaths}. Finally, at the end of the section we prove Lemma~\ref{Lem:Absorbing_Paths} itself by making an appropriate random selection of such copies of $F$ (with the vertices of $F_A$ removed). However, for each of these steps counting unlabelled copies of $F$ presents certain notational inconvenience, and to avoid these we instead count injective maps from $V(F)$ to $V(H)$ which embed $F$ in $H$. To this end, for the rest of this section fix an identification of the vertices of $V(F)$ with the integers $\{1, \dots, L+(k-\ell)\}$ such that the vertices of $F_A$ correspond to the integers $\{L+1, \dots, L+(k-\ell)\}$. Then, given a $k$-graph $H$, we can think of maps $\pi: [L+(k-\ell)] \to V(H)$ as potential embeddings of $F$ in $H$, and count appropriate families of such maps.

Given a $k$-graph $H$ on $n$ vertices, say that an ordered $(k-\ell)$-tuple $A = (a_1, \dots, a_{k-\ell})$ of vertices of $H$ is \emph{$(\gamma, F_{\reg})$-extensible} in $H$ if there are at least $\gamma n^L$ injective maps $\pi: [L+(k-\ell)] \to V(H)$ for which $\pi(L+i) = a_i$ for each $i \in [k-\ell]$ and $\pi(e) \in H$ for every $e \in F_{\reg}$ (in other words, $\pi$ embeds $F_{\reg}$ in $H$ so that the vertices of $F_A$ correspond to the vertices of $A$). We now present an extension lemma which states that our minimum degree condition on $H$ ensures that almost all ordered $(k-\ell)$-tuples of vertices of $H$ are $(\gamma, F_{\reg})$-extensible in $H$.

\begin{lemma} \label{extension}
Fix integers $2 \leq \ell < k$, and suppose that $1/n \ll \gamma \ll \eta \ll \alpha$. If $H$ is a $k$-graph on~$n$ vertices with $\delta_{k-\ell}(H) \geq \alpha n^{k-\ell}$, then all but at most $\eta n^{k-\ell}$ ordered $k-\ell$-tuples $S$ of vertices of $H$ are $(\gamma, F_{\reg})$-extensible in $H$.
\end{lemma}

Lemma~\ref{extension} can be proved by a standard and straightforward application of hypergraph regularity, using the fact that $F_{\reg}$ is a $k$-partite $k$-graph, which implies that we can find many copies of~$F_{\reg}$ within the clusters of an edge of an appropriately-defined reduced $k$-graph. Because this is such a standard application of hypergraph regularity we did not feel it merited the introduction of formal definitions of hypergraph regularity (which are notationally very complex). Instead we simply sketch the key details of the proof.

\begin{proof}[Proof sketch]
Apply the strong hypergraph regularity lemma to $H$ (e.g. this could be the form due to R\"odl and Schacht~\cite{RS}, or the recent regular slice lemma of Allen, B\"ottcher, Cooley and Mycroft~\cite{ABCM} which requires somewhat less notation to be introduced). From this we obtain a partition $\mathcal{C}$ of $V(H)$ into a large constant number $s$ of clusters of equal size $m$ and a reduced $k$-graph $R$ with vertex set $\mathcal{C}$ (so the vertices of $R$ are the clusters) whose edges are all $k$-tuples of clusters $\{X_1, \dots, X_k\} \subseteq \mathcal{C}$ such that, loosely speaking, an appropriate subset of the edges of~$H$ with one vertex in each $X_i$ form a `regular' and dense $k$-partite $k$-graph. We then have the following observations.
\begin{enumerate}[label=(\arabic*)]
\item For almost all ordered $(k-\ell)$-tuples $S$ of vertices of $H$ the vertices of $S$ are contained in~$k-\ell$ distinct clusters of $R$ (this follows by a straightforward counting argument using the fact that there are a large number of clusters of equal size). 
\item For almost all sets $e' \in \binom{\mathcal{C}}{k-\ell}$ of $k-\ell$ clusters of $R$, there is an edge $e \in R$ with $e' \subseteq e$. This is due to the well-known fact (see e.g.~\cite[Lemma~4.3]{KMO}) that the reduced $k$-graph $R$ of $H$ `almost inherits' the minimum degree condition $\delta_{k-\ell}(H) \geq \alpha n^{k-\ell}$, in that almost all $(k-\ell)$-tuples $e'$ of vertices of $R$ have $\deg_R(e') \geq (1-\eps) \alpha s^{k-\ell}$ for a small constant $\eps > 0$.
\item For every edge $e = \{X_1, \dots, X_k\}$ of $R$, almost all ordered $(k-\ell)$-tuples $x = (x_1, \dots, x_{k-\ell})$ of vertices of $H$ with $x_j \in X_j$ for each $j \in [k-\ell]$ extend to many copies of $F_{\reg}$ in which~$x$ plays the role of $F_A$ (in order). Specifically, since $F_{\reg}$ has $L$ vertices outside $F_A$, the number of extensions is at least $c m^L = (c/s^L) n^L$ for some constant $c > 0$. This follows from the extension lemma (for instance in the form proved by Cooley, Fountoulakis, K\"uhn and Osthus~\cite{CFKO}), the fact that $F_{\reg}$ is $k$-partite with $L \leq 3k^2$ vertices, and, crucially, that no edge of $F_{\reg}$ contains more than one vertex of $F_A$.
\end{enumerate} 
Combining these observations immediately yields the lemma.
\end{proof}

We now turn our attention to $F_{\rand}$. Recall for this that $H_p^+ = H \cup H'$, where $H'$ is drawn from $\hnpk$. By exposing the edges of $H'$ in $t$ rounds, we can assume that $H_1 \cup \dots \cup H_t \subseteq H'$, where each $H_t$ is independently drawn from $H_{n, q}^{(k)}$ for $q = p/t$. We take this approach and focus only on the copies of $F_{\rand}$ for which the $i$th edge of $F_{\rand}$ is an edge of $H_i$ for each~$i \in [t]$; by doing so we ensure the independence of key events in our argument. Given a constant proportion of all ordered $L$-tuples of vertices of $H$, the first part of our next lemma gives bounds on how many of these form copies of $F_{\rand}$ in the multi-round process described above, whilst the second part bounds how many of these copies contain a given vertex. Note that together with Lemma~\ref{extension} this shows that almost all $(k-\ell)$-tuples of vertices of $H$ can be extended to many copies of $F$ in this way.

\begin{lemma}\label{numberofrandompaths}
Fix integers $2 \leq \ell < k$, define $t$ and $L$ as in Definition~\ref{defF}, fix $c < 1/t$ and $q \geq \frac{1}{t}n^{-(k-\ell)-c}$, and let $\beta = (1-ct)/2$, so $\beta > 0$. Arbitrarily label the edges of $F_\rand$ as $e_1, \dots, e_t$. Let $V$ be a set of $n$ vertices, and let $\Pi$ be the set of all injective functions $\pi: [L] \to V$. Also let $H_1, \dots, H_t$ be drawn independently from $\hnqk$ on $V$, and define $\Pi^* \subseteq \Pi$ to be the set of all~$\pi \in \Pi$ for which $\pi(e_i) \in H_i$ for every $i \in [t]$. Finally, for each $x \in [L]$ and each $v \in V$ let $\Pi^x_v$ be the set of all $\pi \in \Pi$ with $\pi(x) = v$. Then for every $\gamma > 0$ and every $\Pi' \subseteq \Pi$ of size $|\Pi'| \geq \gamma n^L$ we have
$$\Pr\left(\frac{\gamma q^t n^L}{2^t} \leq |\Pi^* \cap \Pi'| \leq 2^tq^tn^L\right) = 1- \exp\left(-\Omega(n^{2\beta})\right),$$
and for every $v \in V$ and $x \in [L]$ we have
$$\Pr\left(\frac{q^t n^{L-1}}{2^{t+1}} \leq |\Pi^* \cap \Pi^x_v| \leq 2^tq^t n^{L-1} \right) = 1-\exp\left(-\Omega(n^{\beta})\right).$$
\end{lemma}

\begin{proof}
We begin with the first statement. So fix $\gamma > 0$ and a subset $\Pi' \subseteq \Pi$ of size $|\Pi'| \geq \gamma n^L$. For each $0 \leq i \leq t$ define $\Pi_i$ to be the set of injective functions $\pi \in \Pi$ for which $\pi(e_j) \in H_j$ for every $j \in [i]$. We expose the random $k$-graphs $H_1, \dots, H_t$ one by one and show by induction on~$i$ that for each $0 \leq i \leq t$, with probability at least $1 - \exp\left(-\Omega(n^{2\beta})\right)$ we have
\begin{equation} \label{eq:sizeofpi'}
\frac{\gamma q^i n^L}{2^{i}} \leq |\Pi_i \cap \Pi'| \leq  2^i q^i n^L.
\end{equation}
Since $\Pi_t = \Pi^*$, the case $i=t$ will then prove the statement. Observe first that~\eqref{eq:sizeofpi'} holds for $i=0$ with certainty (no random edges are involved in this statement). So suppose now that inequality~\eqref{eq:sizeofpi'} holds for some $i < t$ with probability at least $1- \exp\left(-\Omega(n^{2\beta})\right)$, and observe that the elements of $\Pi_{i+1}$ are precisely those $\pi \in \Pi_i$ with $\pi(e_{i+1}) \in H_{i+1}$. 
Write $\mathcal{S}:= \binom{V}{k}$, and for each $S \in \mathcal{S}$ let $A_S$ count the number of injections $\pi \in \Pi_i \cap \Pi'$ with $\pi(e_{i+1}) = S$.
Then $\sum_{S \in \mathcal{S}} A_S = |\Pi_i \cap \Pi'|$. Moreover for each $S \in \mathcal{S}$ we have $A_S \leq k!n^{L-k}$ since the images of the vertices in $e_{i+1}$ are fixed up to permutation.
Now expose the edges of $H_{i+1}$ and let $A$ count the number of number of injections $\pi \in \Pi_i \cap \Pi'$ with $\pi(e_{i+1}) \in H_{i+1}$. Then $\E(A) = q\sum_{S \in \mathcal{S}} A_S$, so by Proposition~\ref{chernoffbound}, and using~\eqref{eq:sizeofpi'} for the first and final inequalities, we have  
\begin{equation} \label{chernoffcount}
\frac{\gamma q^{i+1} n^L}{2^{i+1}} \leq \frac{q|\Pi_i \cap \Pi'|}{2} = \frac{q \sum_{S \in \mathcal{S}} A_S}{2} \leq A \leq 2 q \sum_{S \in \mathcal{S}} A_S = 2q|\Pi_i \cap \Pi'| \leq 2^{i+1}q^{i+1} n^L
\end{equation}
with probability 
$$ 1-2\exp\left(-\frac{\E(A)}{8k!n^{L-k}}\right)
\geq 1 - 2\exp\left(-\frac{\gamma q^{i+1} n^L}{2^{i+1} \cdot 4k!n^{L-k}}\right) 
\geq 1- \exp\left(-\Omega(q^tn^k)\right).$$
Since $t(k-\ell) = \left(\kFrac-1\right)(k-\ell) \leq k-1$ we have $q^tn^k \geq n^{k-t(k-\ell)-tc} \geq n^{1-tc} \geq n^{2\beta}$. It follows that~\eqref{eq:sizeofpi'} holds for $i+1$ with probability $1 - \exp\left(-\Omega(n^{2\beta})\right)$, completing the induction for~\eqref{eq:sizeofpi'}, and so completing the proof of the first part of the lemma.

\medskip
We now show how to modify the above argument to prove the second part of the statement.
Fix vertices $v \in V$ and $x \in [L]$, and assume for the moment that $t \geq 2$ and $x \in e_1 \cap e_t$. Define $r := |e_1 \cap e_t| \leq \ell$ and, given $R \in \binom{V}{r}$ and $S \in \binom{V}{k-r}$, let $\Phi_{R, S} \subseteq \Pi^x_v$ be the set of injections $\pi \in \Pi^x_v$ with $\pi(e_1 \cap e_t) = R$ and $\pi(e_t \sm e_1) = S$ and set $\Phi'_{R, S} = \Phi_{R, S} \cap \Pi_1$.  
Moreover, for each $S' \in \binom{V}{k-r}$ let $\Phi_{R, S, S'}$ be the set of injections $\pi \in \Pi^x_v$ with $\pi(e_1 \cap e_t) = R$, $\pi(e_t \sm e_1) = S$ and $\pi(e_1 \sm e_t) = S'$, and let $\Phi' _{R, S, S'} = \Phi_{R, S} \cap \Pi_1$.

Observe that for any pairwise-disjoint $R, S$ and $S'$ we have 
$$|\Phi_{R, S, S'}| = m := (r-1)! (k-r)! (k-r)! \frac{(n-2k+r)!}{(n-L)!} = \Theta(n^{L-2k+r}),$$
whilst $|\Phi_{R, S, S'}| = 0$ if $R, S$ and $S'$ are not pairwise-disjoint.
Also, $|\Phi'_{R, S}| = \sum_{S' \in \binom{V}{k-r}} |\Phi'_{R, S, S'}|$. In other words $|\Phi'_{R, S}|$ is the sum of a set of $\binom{n-k-r}{k-r}$ independent random variables which each take value $m$ with probability $q$ and value zero otherwise. So 
$$\E(|\Phi'_{R, S}|) = qm\binom{n}{k-r} = \Theta(qn^{L-k})$$
and we may apply Proposition~\ref{binomchernoffbound} to obtain
\begin{align*} 
\Pr \left( |\Phi'_{R, S}| < qn^{L-k+\beta} \right) 
&\geq 1 - \exp \left(-\Omega \left(\frac{qn^{L-k+\beta}}{n^{L-2k+r}}\right)\right)\\
&\geq 1-\exp\left(-\Omega(n^{\ell-r+\beta})\right)\geq 1-\exp\left(-\Omega(n^{\beta})\right).
\end{align*}
So we may take a union bound over all $R$ and $S$ to find that with probability $1-\exp\left(-\Omega(n^{\beta})\right)$ we have 
\begin{equation}\label{eq:sizephi'} 
|\Phi'_{R, S}| \leq qn^{L-k+\beta}
\end{equation}
for every $R \in \binom{V}{r}$ and $S \in \binom{V}{k-r}$.

We now return to the general case (i.e. we no longer assume that $t \geq 2$ and $x \in e_1 \cap e_t$), where we have two possibilities: either $x \notin \bigcap_{e \in F_{\rand}} e$, in which case by relabelling the edges $e_i$ if necessary we may assume $x \notin e_t$, or $x \in \bigcap_{e \in F_{\rand}} e$, in which case we certainly have $x \in e_1 \cap e_t$. By a similar argument to that used for the first statement of the lemma, we prove by induction on $i$ that for each $0 \leq i \leq t$ with probability at least $1 - \exp\left(-\Omega(n^{\beta})\right)$ we have 
\begin{equation} \label{eq:sizeofpiv}
\frac{q^i n^{L-1}}{2^{i+1}} \leq |\Pi_i \cap \Pi^x_v| \leq 2^i q^i n^{L-1}.
\end{equation}
Again the statement holds with certainty for $i=0$, whilst establishing the case $i=t$ will complete the proof. So suppose that inequality~\eqref{eq:sizeofpiv} holds for some $i < t$ with probability at least $1- \exp\left(-\Omega(n^{\beta})\right)$. For each $S \in \mathcal{S}$ let $B_S$ count the number of injections $\pi \in \Pi_i \cap \Pi^x_v$ with $\pi(e_{i+1}) = S$, so $\sum_{S \in \mathcal{S}} B_S = |\Pi_i \cap \Pi^x_v|$. Define $M := \max_{S \in \mathcal{S}} B_S$, so in all cases we have $M \leq k!n^{L-k}$. Then Proposition~\ref{chernoffbound} and an essentially identical calculation to~\eqref{chernoffcount} yield that~\eqref{eq:sizeofpiv} holds for $i+1$ with probability $p^*$, where
$$ p^* \geq 1-2\exp\left(-\frac{\E(B)}{8M}\right)
\geq 1 - \exp\left(-\Omega\left(\frac{q^{i+1} n^{L-1}}{M}\right)\right).$$
Recall that $t(k-\ell) \leq k-1$, so $k-1 - (t-1)(k-\ell) \geq k-\ell \geq 1$. For $i \leq t-2$, using this and the fact that $M \leq k!n^{L-k}$ and $qn < 1$, we then obtain
$$p^* \geq 1-\exp\left(-\Omega(q^{t-1} n^{k-1})\right) \geq 1-\exp\left(-\Omega(n^{k-1-(t-1)(k-\ell) -tc})\right) \geq 1-\exp\left(-\Omega(n^{2\beta})\right).$$
This leaves only the case when $i = t-1$. Suppose first that $t=1$, in which case using the fact that $M \leq k!n^{L-k}$
and our assumption that $\ell \geq 2$ we obtain 
$$p^* \geq 1-\exp\left(-\Omega(q n^{k-1}) \right) \geq 1-\exp\left(-\Omega(n^{k-1-(k-\ell)-c})\right) \geq 1-\exp\left(-\Omega(n^{2\beta})\right).$$
Next suppose that $x \notin e_t$ then we must have $M \leq k!n^{L-k-1}$, since in counting $B_S$ the image of $x$ is fixed, as well as that of $e_t$. We therefore again have 
$$p^* \geq 1-\exp\left(-\Omega(q^{t} n^{k})\right) \geq 1-\exp\left(-\Omega(n^{2\beta})\right),$$
where the second inequality holds by the same calculation as for the first statement of the lemma.
Finally suppose that $x \in e_1 \cap e_t$, in which case~\eqref{eq:sizephi'} yields $|M| \leq \theta(qn^{L-k+\beta})$, giving 
$$p^* \geq 1-\exp\left(-\Omega(q^{t-1} n^{k-1-\beta})\right) \geq 1-\exp\left(-\Omega(n^{\beta})\right).$$
So in all cases it follows that~\eqref{eq:sizeofpiv} holds for $i+1$ with probability $1 - \exp\left(-\Omega(n^{\beta})\right)$, completing the induction for~\eqref{eq:sizeofpiv}, and so completing the proof of the lemma.
\end{proof}

We are now ready to prove Lemma~\ref{Lem:Absorbing_Paths}.

\begin{proof}[Proof of Lemma~\ref{Lem:Absorbing_Paths}]
Fix integers $2 \leq \ell < k$, define $t:= \kFrac - 1$ and fix $c < 1/t$ and $\beta = 1-ct$, so $\beta > 0$. Introduce constants with $\xi \ll \zeta \ll \gamma \ll \eta \ll\alpha, 1/k$ and suppose that $n$ is sufficiently large. Let $L$, $F$, $F_{\rand}, F_{\reg}$ and $F_A$ be as in Definition~\ref{defF}, arbitrarily label the edges of $F_{\rand}$ as $e_1, \dots, e_t$, and let $H$ be a $k$-graph on $n$ vertices with $\delta_{k-\ell}(H) \geq \alpha n^{\ell}$.

Let $\mathcal{G}$ be the set of all ordered $(k-\ell)$-tuples of distinct vertices of $H$ which are $(\gamma, F_{\reg})$-extensible in $H$. Also define $\mathcal{B} \subseteq \binom{V(H)}{k-\ell}$ to consist of all $(k-\ell)$-sets $B$ for which some ordering of $B$ is not in $\mathcal{G}$. Since by Lemma~\ref{extension} at most $\eta^3 n^{k-\ell}$ ordered $(k-\ell)$-tuple of vertices of $H$ are not members of $\mathcal{G}$, we then have $|\mathcal{B}| \leq \eta^3 n^{k-\ell}$. 

Now fix a probability $p \geq n^{-(k-\ell)-c}$, define $q := p/t$ and let $H_1, \dots, H_t$ be independently drawn from $\hnqk$. Recall that $H_p^+ = H \cup H'$, where $H'$ is drawn from $\hnpk$, so we may assume that $H_i \subseteq H' \subseteq H_p^+$ for each $i \in [t]$. Let $\Pi$ be the set of all injective functions $\pi: [L] \to V(H)$, let $\Pi^* \subseteq \Pi$ be the set of all $\pi \in \Pi$ for which $\pi(e_i) \in H_i$ for every $i \in [t]$, and for each $x \in [L]$ and each $v \in V(H)$ let $\Pi^x_v$ be the set of all $\pi \in \Pi$ with $\pi(x) = v$. By Lemma~\ref{numberofrandompaths} we then have $q^tn^L/2^{t+1} \leq |\Pi^*| \leq 2^tq^tn^L$ with probability $1-\exp(-\Omega(n^{2\beta}))$, and for every $x \in [L]$ and $v \in V(H)$ we have $|\Pi^* \cap \Pi_v^x| \leq 2^tq^tn^{L-1}$ with probability $1-\exp(-\Omega(n^{\beta}))$. 

For every $A = (a_1, \dots, a_{k-\ell}) \in \mathcal{G}$ the fact that $A$ is $(\gamma, F_{\reg})$-extensible means that there are at least $\gamma n^L$ injective maps $\pi: [L+(k-\ell)] \to V(H)$ for which $\pi(L+i) = a_i$ for each $i \in [k-\ell]$ and $\pi(e) \in H$ for every $e \in F_{\reg}$. For each such $\pi$ let $\pi'$ be the restriction of $\pi$ to $[L]$, and let $\Pi_A$ denote the set of all maps $\pi'$ formed in this way. Since the images of the vertices $L+1, \dots, L+(k-\ell)$ were fixed in $\pi$ each such $\pi'$ is distinct, so we have $|\Pi_A| \geq \gamma n^L$. We may therefore apply Lemma~\ref{numberofrandompaths} to find that for each $A \in \mathcal{G}$ we have $|\Pi^* \cap \Pi_A| \geq \gamma q^tn^L/2^t$ with probability at least $1-\exp(-\Omega(n^{2\beta}))$.

Since there are at most $nL$ pairs $(v, x)$ with $v \in V(H)$ and $x \in [L]$, and $|\mathcal{G}| \leq n^{k-\ell}$, we may take a union bound to find that with high probability we have
\begin{enumerate}[label=(\alph*), noitemsep]
\item $q^tn^L/2^{t+1} \leq |\Pi^*| \leq 2^tq^tn^L$,
\item $|\Pi^* \cap \Pi_v^x| \leq 2^tq^tn^{L-1}$ for every $v \in V(H)$ and $x \in L$, and
\item $|\Pi^* \cap \Pi_A| \geq \gamma q^tn^L/2^t$ for every $A \in \mathcal{G}$.
\end{enumerate}
It therefore suffices to show that that if~(a),~(b) and~(c) all hold then we can find a set $\paths$ as in the statement of the lemma. To do this, choose a random subset $\Phi \subseteq \Pi^*$ by including each $\pi \in \Pi^*$ in $\Phi$ with probability $\zeta n/|\Pi^*|$, independently of all other choices. Then we have $\E (|\Phi|)=\zeta n$, and using~(a) and~(c) we find that for every $A \in \mathcal{G}$ we have $\E(|\Pi_A \cap \Phi|) = \zeta n |\Pi_A \cap \Pi^*|/|\Pi^*| \geq \zeta \gamma n/4^t$.
Applying a Chernoff bound we find that with probability $1-o(1)$ we have 
\begin{equation} \label{phibounds} 
|\Phi| \leq 2\zeta n \mbox{ and } |\Pi_A \cap \Phi| \geq  \frac{\zeta \gamma n}{4^{t+1}} \mbox{ for every $A \in \mathcal{G}$.}
\end{equation}

Let $\mathcal{I}$ denote the set of ordered pairs $(\pi, \pi')$ with $\pi, \pi' \in \Pi^*$ and $\pi \neq \pi'$ for which the images of $\pi$ and $\pi'$ intersect. Then for any $\pi \in \Pi^*$, the number of maps $\pi' \in \Pi^*$ with $(\pi, \pi') \in \mathcal{I}$ is at most $\sum_{x,y \in [L]} |\Pi^* \cap \Pi^x_{\pi(y)}|$ (we do not have equality as maps which intersect $\pi$ in more than one vertex are overcounted). Using~(a) and~(b) it follows that $|\mathcal{I}| \leq |\Pi^*| \cdot L^2 \cdot 2^t q^t n^{L-1} \leq 2^{2t+1}L^2|\Pi^*|^2/n$. So, defining $Z$ to be the number of pairs $(\pi, \pi') \in \mathcal{I}$ with $\pi, \pi' \in \Phi$, we have 
$$\E(Z) = \frac{2^{2t+1}L^2|\Pi^*|^2}{n} \cdot \left(\frac{\zeta n}{|\Pi^*|}\right)^2 = 2^{2t+1}L^2 \zeta^2 n$$
and so by Markov's inequality the event $Z \leq  4^{t+1} L^2 \zeta^2 n$ has probability at least $1/2$.

We may therefore fix a subset $\Phi \subseteq \Pi^*$ for which $Z \leq 4^{t+1} L^2 \zeta^2 n$ and such that~\eqref{phibounds} holds. Having done this, we form a subset $\Phi' \subseteq \Phi$ as follows. First, for every $(\pi, \pi') \in \mathcal{I}$ with $\pi, \pi' \in \Phi$ we remove both $\pi$ and $\pi'$ from $\Phi$; this results in at most $Z$ elements of $\Phi$ being removed. Second, remove any $\pi \in \Phi$ which is not in $\bigcup_{A \in \mathcal{G}} \Pi_A$. Using~\eqref{phibounds} we then have for every $A \in \mathcal{A}$ that
\begin{equation}\label{eqn:number_of_absorbing_paths_for_a_tuple}
|\Pi_A \cap \Phi'| \geq |\Pi_A \cap \Phi| - Z \geq  \frac{\zeta\gamma n}{4^{t+1}} - 4^{t+1} L^2 \zeta^2 n \geq \xi n. 
\end{equation}

Recall that $F$ contains an $\ell$-path $P(F)$ with vertex set $V(F) \sm F_A$ and ends $F^{\pbeg}$ and $F^{\pend}$. Note that since $\Phi' \subseteq \Phi \subseteq \Pi^*$, every $\pi \in \Phi'$ embeds $F_{\rand}$ in $H_p^+$. Furthermore, our choice of~$\Phi'$ implies that for every $\pi \in \Phi'$ we have $\pi \in \Pi_A$ for some $A = (a_1, \dots, a_{k-\ell}) \in \mathcal{G}$. This means that the extension of $\pi$ to $[L+k-\ell]$ obtained by setting $\pi(L+i) = a_i$ for each $i \in [k-\ell]$ is an embedding of $F_\reg$ in $H$, and also an embedding of $F$ in $H_p^+$ (the latter due to our previous observation that $\pi$ embeds $F_{\rand}$ in $H_p^+$). In particular, since $P(F) \subseteq F_{\reg} \sm F_A$, the image of $P(F)$ under $\pi$ is an $\ell$-path in $H[\pi([L])]$, which we denote by $P_\pi$. Define $\paths := \{P_\pi : \pi \in \Phi'\}$, so $|\paths| \leq |\Phi| \leq  2 \zeta n$ by~\eqref{phibounds}. Observe also that the paths in $\paths$ are vertex-disjoint by our choice of $\Phi'$, and that (a) is satisfied since each path in $\paths$ has $L \leq 3k^2$ vertices. Since $F$ also contains an $\ell$-path $Q(F)$ with vertex set $V(F)$ and ends $F^{\pbeg}$ and $F^{\pend}$, it follows that for every $A \in \mathcal{G}$ and every $\pi \in \Pi_A \cap \Phi'$ the path $P_\pi$ can absorb $A$ in $H_p^+$. So~\eqref{eqn:number_of_absorbing_paths_for_a_tuple} ensures that condition (c) is satisfied. 

It remains to satisfy condition (b). Recall for this that $|\mathcal{B}| \leq \eta^3 n^{k-\ell}$. Let $B$ be the set of of all vertices $v \in V(H)$ which are contained in more than $\eta n^{k-\ell-1}$ members of $|\mathcal{B}|$, so $|B| \leq \eta n/2$. Greedily choose for each $v \in B$ an edge $e(v)$ of $H$ which does not intersect $V(\paths) := \bigcup_{P \in \paths}V(P)$ or any edge $e(u)$ chosen for some previous $u \in B$. To see that this is possible, note that $V(\paths)$ and the previously-chosen edges cover at most $L |\paths| + k|B| \leq L 2\zeta n + k \eta n/2 \leq k \eta n$ vertices, and so at most $k \eta n^{k-1}$ edges contain $v$ and a vertex of either $V(\paths)$ or a previously-chosen $e(u)$. On the other hand, since $\delta_{k-\ell}(H) \geq \alpha n^{k-\ell}$, at least $\alpha \binom{n-1}{k-1} > k \eta n^{k-1}$ edges of $H$ contain $v$, so there is at least one feasible choice for $e(v)$. We add the edges $e(v)$ for each $v \in B$ to $\paths$; after doing this $\paths$ is a collection of at most $2 \zeta n + \eta n/2 \leq \eta n$ vertex-disjoint $\ell$-paths in $H$ (since each edge $e(v)$ is an $\ell$-path of length one in $H$), and the addition of these edges does not affect the validity of (a) or (c). However, since we now have $v \in V(\paths)$ for every $v \in B$, condition (b) is satisfied also, completing the proof.
\end{proof}

\section{The connecting lemma} \label{sec:connecting}
In this section we prove our connecting lemma, Lemma~\ref{Lemma:ConnectingPaths}. 
We begin by describing the connecting structures we will use. 
Fix $2 \leq \ell < k$, and define 
$$t := \left\lceil \frac{k}{k-\ell} \right\rceil - 1 \mbox{ and } b := 3(k-\ell)t+k.$$
Let $P$ be an $\ell$-path of length $3t+1$ with vertex set $[b]$, with vertices labelled in a natural order, that is, so that the edges of $P$ are $e_i := \{(k-\ell)i+1, \dots, (k-\ell)i+k\}$ for each $0 \leq i \leq 3t$.  For each $0 \leq i \leq t$ define $F_i$ to be the $k$-graph with vertex set $[b]$ and edge set
$$ E(F_i) = \{e_0, e_1, \dots, e_t\} \cup \{e_{2t+1-i}, \dots, e_{2t}\} \cup \{e_{2t+1}, \dots, e_{3t}\}.$$
In particular~$F_0$ consists of two vertex-disjoint $\ell$-paths, one of length $t$ and one of length $t+1$. To see this observe that each edge contains $k-\ell$ vertices not in the previous edge, so the first vertex of $e_{2t+1}$ is $(t+1)(k-\ell) \geq k$ vertices subsequent to the first vertex of $e_t$. We note for future reference that the number of isolated vertices in $F_0$ is $b - (2t+1)(k-\ell)-2\ell = t(k-\ell)-\ell$. Observe also that $F_t$ is precisely the path~$P$, whilst for each $i \in [t]$ we form $F_{i}$ from $F_{i-1}$ by adding the edge $e_{2t+1-i}$. Define ordered $\ell$-tuples $F^{\pbeg} := (1, \dots, \ell)$ and $F^{\pend} := (b-\ell+1, \dots, b)$; we refer to these as the \emph{ends} of each~$F_i$ (note that this coincides with the definition of ends of $P$). Similarly we define $F^{\pint} := \{\ell+1, \dots, b-\ell\}$, and refer to the vertices of $F^{\pint}$ as \emph{interior vertices} of each $F_i$.

In our proof of Lemma~\ref{Lemma:ConnectingPaths} we will iteratively connect paths together until only a constant number of paths remain. The main difficulty is then to connect this constant number of paths to form a cycle. We achieve this by the following lemma, which essentially states that Lemma~\ref{Lemma:ConnectingPaths} holds for a large constant number of paths (by taking the ends of these paths as the input $\ell$-tuples). Similarly as in the previous section, to prove this lemma it is notationally convenient to count injective functions from $V(P)$ to $V(H)$ rather than copies of $P$ in $H$.

\begin{lemma}\label{connectfewpaths}
Fix integers $2 \leq \ell < k$, define $t := \kFrac - 1$ and $b := 3(k-\ell)t+k$, and fix a constant $c < 1/t$. Suppose that $\beta \ll \alpha, 1/k$. Let $H$ be a $k$-graph on~$n$ vertices and for some $s \leq 1/\beta$ let $a_1, \dots, a_s$ and $b_1, \dots, b_s$ be $2s$ ordered $\ell$-tuples of vertices of $H$ such that the sets $a_r \cup b_r$ for $r \in [s]$ are pairwise vertex-disjoint. 
Suppose that, for some set $X \subseteq V(H) \sm \bigcup_{r \in [s]} (a_r \cup b_r)$, 
every set $S \in \binom{V(H)}{\ell}$ satisfies $\deg_H(S, X) \geq \alpha n^{k-\ell}$. If $p \geq n^{-(k-\ell)-c}$, then with high probability there exist pairwise vertex-disjoint $\ell$-paths $Q_1, \dots, Q_s$ in $H_p^+$ with $|\bigcup_{r \in [s]} V(Q_r)| \leq bs$ and so that
for each~$r \in [s]$ the path $Q_r$ has $Q_r^{\pint} \subseteq X$ and ends $Q_r^{\pbeg} = a_r$ and $Q_r^{\pend} = b_r$.
\end{lemma}

\begin{proof}
We begin with the following claim, which shows that for each $r \in [s]$ there are many copies of $F_0$ in $H$ with ends $a_r$ and $b_r$ whose interior vertices lie in $X$.

\begin{claim} \label{claim1}
For every $r \in [s]$ there are at least $\tfrac{1}{2}(\alpha n)^{b-2\ell}$ injective functions $\pi : [b] \to V(H)$ for which $\pi$ is an embedding of $F_0$ in $H$ with $\pi(F^{\pbeg}) = a_r$, $\pi(F^{\pend}) = b_r$ and $\pi(F^{\pint}) \subseteq X$. 
\end{claim}

To prove Claim~\ref{claim1}, we consider the possible ways to form a (not necessarily injective) function $\pi : [b] \to V(H)$ with $\pi(F^{\pbeg}) = a_r$, $\pi(F^{\pend}) = b_r$ and $\pi(F^{\pint}) \subseteq X$ and such that $\pi(e) \in H$ for every $e \in F_0$. First, for each $i \in [\ell]$ we set $\pi(i)$ to be the $i$th vertex of $a_r$, so $\pi(F^{\pbeg}) = a_r$ as required. Now observe that there are $k-\ell$ vertices in $e_0$ which are not in $F^{\pbeg}$. Since $\deg_H(a_r, X) \geq \alpha n^{k-\ell}$ by assumption, there are at least $\alpha n^{k-\ell}$ possible ways to choose $\pi(e_0 \sm F^{\pbeg})$ in $X$ so that $\pi(e_0)$ is an edge of $H$. In exactly the same way we find that for each $1 \leq i \leq t$, having fixed $\pi(e_0 \cup \dots \cup e_{i-1})$, since $\deg_H(\pi(e_{i-1}\cap e_i), X) \geq \alpha n^{k-\ell}$, there are at least $\alpha n^{k-\ell}$ ways to choose $\pi(e_i \sm e_{i-1})$ in $X$ so that $\pi(e_i)$ is an edge of $H$. Overall this gives us at least $(\alpha n^{k-\ell})^{t+1}$ ways to choose the images of the vertices covered by $e_0, \dots, e_t$. 

Next, for each $i \in [\ell]$ we set $\pi(b-\ell+i)$ to be the $i$th vertex of $b_r$, so $\pi(F^{\pend}) = b_r$, as required. Similarly as before there are then at least $\alpha n^{k-\ell}$ possible ways to choose $\pi(e_{3t} \sm F^{\pend})$ in $X$ so that $\pi(e_{3t})$ is an edge of $H$, and for each $2t+1 \leq i \leq 3t-1$, having fixed $\pi(e_{3t} \cup \dots \cup e_{i+1})$, there are at least $\alpha n^{k-\ell}$ ways to choose $\pi(e_i \sm e_{i+1})$ in $X$ so that $\pi(e_i)$ is an edge of $H$. Overall this gives us at least $(\alpha n^{k-\ell})^t$ ways to choose the images of the vertices covered by $e_{2t+1}, \dots, e_{3t}$. 

Finally there are $t(k-\ell) - \ell$ isolated vertices in $F_0$, and for each such vertex $v$ there are at least $\alpha n$ possibilities for $\pi(v)$ in $X$ (the fact that $|X| \geq \alpha n$ follows from our degree assumption). We conclude that in total the number of functions $\pi : [b] \to V(H)$ with $\pi(F^{\pbeg}) = a_r$, $\pi(F^{\pend}) = b_r$ and $\pi(F^{\pint}) \subseteq X$ such that $\pi(e) \in H$ for every $e \in F_0$ is at least
$$(\alpha n)^{(t+1)(k-\ell)+t(k-\ell)+(t(k-\ell) - \ell)} = (\alpha n)^{3t(k-\ell)+k-2\ell} =(\alpha n)^{b-2\ell}.$$
Every such function $\pi$ which is injective is an embedding of $F_0$ in $H$ with the properties described in the statement of the claim. Since the number of non-injective functions $\pi : [b] \to V(H)$ with $\pi(F^{\pbeg}) = a_r$ and $\pi(F^{\pend}) = b_r$ is at most $\binom{b}{2}n^{b-2\ell-1} \leq \tfrac{1}{2}(\alpha n)^{b-2\ell}$, the claim follows. \medskip

Our next claim uses an inductive argument to show that with high probability we can find a copy of $P$ in $H_p^+$ with prescribed ends and avoiding a given set of vertices. The proof of this claim is broadly similar to that of Lemma~\ref{numberofrandompaths}, but here we have $2\ell$ fixed vertices of the embedding (the ends of~$P$).

\begin{claim} \label{claim2}
Given $r \in [s]$ and a set $Z \subseteq X$ of size $|Z| \leq bs$, with probability $1-\exp(-\Omega(n^{1-tc}))$ there is a copy of $P$ in $H_p^+$ with ends $a_r$ and $b_r$ whose interior vertices all lie in~$X \sm Z$.
\end{claim}

To prove Claim~\ref{claim2} we use a multiple exposure argument with $t$ rounds. So let $q := p/t$ and for each $i \in [t]$ let $H_i$ be drawn independently from $\hnqk$. We may then assume that $H \cup \bigcup_{i \in [t]} H_i \subseteq H_p^+$. 

For each $0 \leq i \leq t$ define $\Pi_i$ to be the set of injective functions $\pi : [b] \to V(H)$ for which~$\pi$ is an embedding of $F_i$ in $H \cup \bigcup_{j \in [i]} H_j$ such that $\pi(F^{\pbeg}) = a_r$, $\pi(F^{\pend}) = b_r$, and $\pi(F^{\pint}) \subseteq X \sm Z$. We prove by induction on $i$ that, for each $0 \leq i \leq t$, with probability at least $1- \exp\left(-\Omega(n^{1-tc})\right)$ we have
\begin{equation} \label{eq:sizeofpi}
|\Pi_i| \geq \frac{(\alpha n)^{b-2\ell}q^i}{4\cdot 2^{i}}
\end{equation} 

Our base case $i = 0$ follows immediately from Claim~\ref{claim1}. Indeed, every injective function $\pi$ as in the statement of Claim~\ref{claim1} is an element of $\Pi_0$ unless the image of $\pi$ contains a vertex in~$Z$. However, since $|Z| \leq bs$, there are at most $b^2s n^{b-2\ell-1} \leq \tfrac{1}{4}(\alpha n)^{b-2\ell}$ functions $\pi : [b] \to V(H)$ with $\pi(F^{\pbeg}) = a_r$ and $\pi(F^{\pend}) = b_r$ which include a vertex of $Z$ in their image; this gives the desired inequality (with certainty). 

Suppose therefore that for some $i < t$ inequality~\eqref{eq:sizeofpi} holds with probability at least $1- \exp\left(-\Omega(n^{1-tc})\right)$, and observe that for any injection $\pi \in \Pi_i$ with $\pi(e_{2t-i}) \in H_{i+1}$ we have $\pi \in \Pi_{i+1}$. Write $\mathcal{S}:= \binom{V(H)}{k}$, and for each $S \in \mathcal{S}$ let $N_S$ denote the number of injections $\pi \in \Pi_i$ with $\pi(e_{2t-i}) = S$. 
Then $\sum_{S \in \mathcal{S}} N_S = |\Pi_i|$, and $N_S \leq k!n^{b - 2\ell - k}$ for each $S \in \mathcal{S}$ since the images of the vertices in $e_{2t-i} \cup F^{\pbeg} \cup F^{\pend}$ are fixed up to permutations of $e_{2t-i}$. 
Exposing the edges of $H_{i+1}$, let $N$ be the number of injections $\pi \in \Pi_i$ with $\pi(e_{2t-i}) \in H_{i+1}$. Then by Proposition~\ref{chernoffbound} and using our inductive hypothesis for the final inequality, we have  
$$N \geq \frac{\E(N)}{2} = \frac{q \sum_{S \in \mathcal{S}} N_S}{2} =  \frac{q|\Pi_i|}{2} \geq \frac{(\alpha n)^{b-2\ell}q^{i+1}}{4 \cdot 2^{i+1}}$$ 
with probability  
\begin{align*}
1- \exp\left(-\frac{2(\alpha n)^{b-2\ell}q^{i+1}}{4 \cdot 2^{i+1} \cdot 4k!n^{b-2\ell-k}}\right) 
\geq 1- \exp\left(-\Omega(p^tn^k)\right).
\end{align*}
Observe that $t(k-\ell) = \left(\kFrac-1\right)(k-\ell) \leq k-1$, so $p^tn^k \geq n^{k-t(k-\ell)-tc} \geq n^{1-tc}$. It follows that \eqref{eq:sizeofpi} holds for $i+1$ with probability $1 - \exp\left(-\Omega(n^{1-tc})\right)$, completing the induction argument.

In particular, the case $i=t$ of~\eqref{eq:sizeofpi} shows that with probability $1- \exp\left(-\Omega(n^{1-tc})\right)$ we have $\Pi_t > 0$. Since each $\pi \in \Pi_t$ is an embedding of $F_t = P$ in $H \cup \bigcup_{j \in [t]} H_j \subseteq H_p^+$ such that $\pi(F^{\pbeg}) = a_r$, $\pi(F^{\pend}) = b_r$, and $\pi(F^{\pint}) \subseteq X \sm Z$, this demonstrates the existence of a copy of $P$ in $H_p^+$ as in the statement of Claim~\ref{claim2}, and so completes the proof of the claim. \medskip

Returning to the proof of the lemma, say that $H_p^+$ is \emph{well-connected} if for every $r \in [s]$ and every set $Z \subseteq X$ of size $|Z| \leq sb$ there exists a copy of $P$ in $H_p^+$ with ends~$a_r$ and~$b_r$ whose interior vertices all lie in~$X \sm Z$. Taking a union bound over every such~$r$ and~$Z$, by Claim~\ref{claim2} and our assumption that $c < 1/t$ we find that the probability that $H_p^+$ is well-connected at least 
$$ 1 - s \cdot n^{sb} e^{-\Omega(n^{1-tc})} \geq 1 - o(1).$$
In other words, $H_p^+$ is well-connected with high probability. 

Finally, observe that if $H_p^+$ is well-connected then we can greedily 
choose vertex-disjoint $\ell$-paths $Q_r$ for $r \in [s]$ as desired.
Indeed, for each $r \in [s]$ in turn we choose a copy $Q_r$ of $P$ in $H_p^+$ as follows. Let $Z$ consist of all vertices covered by the previously-chosen paths $Q_1, Q_2, \dots, Q_{r-1}$, and choose a copy $Q_r$ of $P$ in $H_p^+$ with ends $a_r$ and $b_r$ in which all interior vertices lie in $X \sm Z$. Since $Z$ consists of at most $b$ vertices from each of at most $s$ previously-chosen copies of $P$, we have $|Z| \leq bs$, and so we may choose $Q_r$ in this way since $H_p^+$ is well-connected. Do this for every $r \in S$ and observe that since each path $Q_r$ has $b$ vertices we then have $|\bigcup_{r \in [s]} V(Q_r)| = bs$, as required.
\end{proof}

Given an ordered $\ell$-tuple $S = (x_1, \dots, x_\ell)$, we define $\overleftarrow{S} = (x_\ell, \dots, x_1)$, so $\overleftarrow{S}$ is the ordered $\ell$-tuple formed by reversing the order of $S$. This definition is motivated by the following observation.
Suppose that $Q_1, Q_2$ and $Q_3$ are $\ell$-paths such that $Q_i$ has ends $a_i := Q_i^{\pbeg}$ and $b_i := Q_i^{\pend}$ for each $i \in [3]$. Observe that if $b_1 = a_2$ and $b_2 = \overleftarrow{b_3}$, and the vertices of $Q_1$, $Q_2$ and $Q_3$ are otherwise distinct, then $Q_1Q_2Q_3$ (the $k$-graph with vertex set $\bigcup_{i \in [3]} V(Q_i)$ and edge set $\bigcup_{i \in [3]} V(Q_i)$) is an $\ell$-path with ends $a_1$ and $\overleftarrow{a_3}$ (think of traversing $Q_1$ and $Q_2$ from $a_1$ via $a_2=b_1$ to $b_2 = \overleftarrow{b_3}$, then traversing $Q_3$ `in reverse' from $\overleftarrow{b_3}$ to $\overleftarrow{a_3}$).

\begin{proof}[Proof of Lemma~\ref{Lemma:ConnectingPaths}]
Fix integers $2 \leq \ell < k$, define $t := \kFrac - 1$ and $b := 3(k-\ell)t+k$ and introduce constants with $\vartheta \ll \beta \ll \alpha, 1/k$. Let $H$ be a $k$-graph on~$n$ vertices and let $\paths$ be a collection of at most $\vartheta n$ vertex-disjoint $\ell$-paths in $H$ such that, writing $X := V(H) \sm \bigcup_{P \in \paths} V(P)$, for every set $S \in \binom{V(H)}{\ell}$ we have $\deg_H(S, X) \geq \alpha n^{k-\ell}$. Also fix $c < 1/t$ and $p \geq n^{-(k-\ell)-c}$. We begin with the following claim.

\begin{claim} \label{claim3}
Let $\mathcal{S}$ be a set of $m$ pairs $(x, y)$ for which both $x$ and $y$ are ordered $\ell$-tuples of vertices of $H$, and suppose that the sets $x \cup y$ for $(x, y) \in \mathcal{S}$ are pairwise vertex-disjoint. 
If $m > 1/\beta$ then there must exist distinct pairs $(a_1, b_1), (a_2, b_2) \in \mathcal{S}$ for which there are at least $\tfrac{1}{2} \beta^4 n^{b-2\ell}$ injective functions $\pi : [b] \to V(H)$ such that $\pi$ is an embedding of $F_1$ in $H$ with $\pi(F^{\pbeg}) = b_1$, $\pi(F^{\pend}) = \overleftarrow{b_2}$ and $\pi(F^{\pint}) \subseteq X$. Moreover, for each set $Z \subseteq X$ of size $|Z| \leq b\vartheta n$, with probability $1-\exp(-\Omega(n^{2-tc}))$ there is a copy of $P$ in $H_p^+$ with ends $b_1$ and $\overleftarrow{b_2}$ whose interior vertices all lie in~$X \sm Z$.
\end{claim}

To prove Claim~\ref{claim3}, consider any $(x, y) \in \mathcal{S}$, and observe that by the same argument as in the proof of Claim~\ref{claim1} there are at least $(\alpha n)^{(t+1)(k-\ell)}$ maps $\phi: e_0 \cup \dots \cup e_{t} \to y \cup X$ such that $\phi(F^{\pbeg}) = y$ and $\phi(e_i) \in H$ for every $0 \leq i \leq t$. Let $r = |e_{t} \cap e_{2t}|$. It then follows that there are at least $3\beta n^{r}$ ordered $\ell$-tuples $S$ of vertices of $V$ which have $\phi(e_{t+1} \cap e_{2t+1}) = S$ for at least $\beta n^{(t+1)(k-\ell)-r}$ such maps $\phi$. Let $\mathcal{S}(x,y)$ denote the set of all such $S$. If $|\paths| > 1/\beta$ then by inclusion-exclusion we can choose distinct $(a_1, b_1), (a_2, b_2) \in \mathcal{S}$ for which $|\mathcal{S}(a_1,b_1) \cap \mathcal{S}(a_2,b_2)| \geq \beta^2 n^{r}$. Then for any $\phi \in \mathcal{S}(a_1,b_1)$ and $\phi' \in \mathcal{S}(a_2,b_2)$ we may define a map $\pi : [b] \to V(H)$ by
$$\pi(x) = \begin{cases}
\phi(x) & \mbox{ if $x \in e_0 \cup \dots \cup e_{t}, $}  \\
\phi'(b+1-x) & \mbox{ otherwise.}
\end{cases}$$
Since no edge of $F_1$ intersects $e_{t} \cap e_{2t}$ other than $e_{t}$ and $e_{2t}$ we then have $\pi(e) \in H$ for every $e \in F_1$. Moreover $\pi(F^{\pbeg}) = b_1$, $\pi(F^{\pend}) = \overleftarrow{b_2}$ and $\pi(F^{\pint}) \subseteq X$. In this way we obtain $$\beta^2 n^{r}\cdot (\beta n^{(t+1)(k-\ell)-r})^2 = \beta^4 n^{b-2\ell}$$ 
such maps $\pi$, since $r = k-t(k-\ell)$ and $b = 3t(k-\ell)+k$. Since there are at most $b^2n^{b-2\ell-1} \leq \tfrac{1}{2} \beta^4 n^{b-2\ell}$ maps $\pi : [b] \to V(H)$ such that $\pi(F^{\pbeg}) = b_1$ and $\pi(F^{\pend}) = \overleftarrow{b_2}$ which are not injective, this completes the proof of the first statement of the claim. 

Now fix $Z \subseteq X$ of size $|Z| \leq b\vartheta n$. Set $H_1 := H$ and let $H_2, \dots, H_t$ be independently drawn from $\hnqk$, where $q = p/k$. We may then assume that $\bigcup_{i \in [t]} H_i \subseteq H_p^+$. For each $i \in [t]$ define $\Pi_i$ to be the set of injective functions $\pi : [b] \to V(H)$ for which~$\pi$ is an embedding of $F_i$ in $\bigcup_{j \in [i]} H_j$ such that $\pi(F^{\pbeg}) = b_1$, $\pi(F^{\pend}) = \overleftarrow{b_2}$, and $\pi(F^{\pint}) \subseteq X \sm Z$. Arguing by induction on $i$  then shows that for each $i \in [t]$, with probability at least $1- \exp\left(-\Omega(n^{2-tc})\right)$ we have
\begin{equation} \label{eq:sizeofpi2}
|\Pi_i| \geq \frac{\beta^4 n^{b-2\ell}q^{i-1}}{2^i}
\end{equation}
Indeed, the argument is almost identical to that for Claim~\ref{claim3}, with the exception that the base case is now the case $i=1$, for which~\eqref{eq:sizeofpi2} holds with certainty by the first part of this claim. The crucial point is that $|\Pi_i|$ is larger by a factor of $\Theta(q^{-1})$ in~\eqref{eq:sizeofpi2} compared to \eqref{eq:sizeofpi}, and so for each $i \in [t]$ we find that~\eqref{eq:sizeofpi2} holds with probability $1- \exp(-\Omega(n^{p^{t-1}n^k}))$ instead of $1- \exp(-\Omega(n^{p^{t}n^k}))$ as before. Since $p^{t-1}n^k \geq 2-tc$, this completes the proof of the claim.
\medskip


Now set $s := |\paths|$, and set $\mathcal{E}_0 := \{(P^{\pbeg}, P^{\pend}) : {P \in \paths}\}$ and $\mathcal{F}_0 : = \emptyset$. We proceed iteratively through $s$ steps to identify how we shall connect the paths in $\paths$ into our desired $\ell$-cycle.
At each step $i\geq 1$, if $|\mathcal{E}_{i-1}| > 1/\beta$, we choose distinct pairs $(a_1, b_1), (a_2, b_2) \in \mathcal{E}_{i-1}$ as in Claim~\ref{claim3}, and set 
$$ \mathcal{E}_i = (\mathcal{E}_{i-1} \cup \{(a_1, \overleftarrow{a_2})\}) \sm \{(a_1, b_1), (a_2, b_2)\} 
\mbox{ and } \mathcal{F}_i = \mathcal{F}_{i-1} \cup \{(b_1, \overleftarrow{b_2})\}.$$
On the other hand, if $2 \leq |\mathcal{E}_{i-1}| \leq 1/\beta$ then we arbitrarily choose  distinct pairs $(a_1, b_1), (a_2, b_2) \in \mathcal{E}_{i-1}$ and set
$$ \mathcal{E}_i = (\mathcal{E}_{i-1} \cup \{(a_1, b_2)\}) \sm \{(a_1, b_1), (a_2, b_2)\} 
\mbox{ and } \mathcal{F}_i = \mathcal{F}_{i-1} \cup \{(b_1, a_2)\}.$$ 
Finally, if $|\mathcal{E}_{i-1}| = 1$ then let $(a_1, b_1)$ be the unique element of $\mathcal{E}_{i-1} = 1$ and set $\mathcal{E}_i = \emptyset$ and $\mathcal{F}_i = \mathcal{F}_{i-1} \cup \{(b_1, a_1)\}$. Since $|\mathcal{E}_0| = s$ and $|\mathcal{E}_i| = |\mathcal{E}_{i-1}| - 1$ for each $i \in [s]$ we find that $\mathcal{E}_s$ is empty, and we terminate the iteration at the end of this step.

Let $z = s- \lfloor 1/\beta \rfloor$, so the final $z$ steps were those at which the pairs were chosen arbitrarily rather than by using Claim~\ref{claim3}. Taking a union bound over all the $z \leq s \leq \theta n$ pairs $(x, \overleftarrow{y}) \in \mathcal{F}_{z}$ and all sets $Z \subseteq X$ of size at most $b \theta n$, we find with high probability that for every pair $(x, \overleftarrow{y}) \in \mathcal{F}_z$ and every set $Z \subseteq X$ of size at most $b \theta n$ there exists a copy $Q_{(x, \overleftarrow{y})}$ of $P$ in $H_p^+$ with ends $Q_{(x, \overleftarrow{y})}^{\pbeg} = x$ and $ Q_{(x, \overleftarrow{y})}^{\pend} = \overleftarrow{y}$ whose interior vertices all lie in $X \sm Z$. Furthermore, Lemma~\ref{connectfewpaths} tells us that with high probability there exists a collection $\mathcal{Q}$ of $s-z$ pairwise vertex-disjoint $\ell$-paths $Q_{(x,y)}$ in $H_p^+$ for each pair $(x, y) \in \mathcal{F}_s \sm \mathcal{F}_{z}$ such that $Q_{(x, y)}$ has ends $Q_{(x,y)}^{\pbeg} = x$ and $Q_{(x, y)}^{\pend} = y$ and so that $Q^{\pint}_{(x, y)} \subseteq X \sm Z$, and moreover that the paths in $\mathcal{Q}$ cover at most $b(s-z)$ vertices. Fix an outcome of our random formation of $H_p^+$ for which both these events occur. Having done this, we choose a copy $Q_{(x, \overleftarrow{y})}$ of $P$ in $H_p^+$ for each pair $(x, \overleftarrow{y}) \in \mathcal{F}_z$ with ends $Q_{(x, \overleftarrow{y})}^{\pbeg} = x$ and $ Q_{(x, \overleftarrow{y})}^{\pend} = \overleftarrow{y}$ and with $Q_{(x, \overleftarrow{y})}^{\pint} = X \sm Z$ so that the chosen copies are pairwise vertex-disjoint and do not intersect any path in $\mathcal{Q}$. Indeed, we can make these choices greedily, at each step taking the set $Z \subseteq X$ to consist of all vertices in $X$ which are contained in a member of $\mathcal{Q}$ or a previously-chosen $Q_{(x, \overleftarrow{y})}$, and we then have $|Z| \leq sb \leq b\theta n$. Having done so, we add all of these chosen copies to $\mathcal{Q}$.

We are now ready to form our desired cycle. For this we initially take $\paths_0 = \paths$, and repeat each step of the iterative process above in turn.  At each step $i \in [z]$ we chose some $(a_1, b_1)$ and $(a_2, b_2)$ in $\mathcal{E}_{i-1}$ and added the pair $(b_1, \overleftarrow{b_2})$ to~$\mathcal{F}_{i-1}$ to form $\mathcal{F}_i$, and we now let $P_1, P_2 \in \paths_{i-1}$ be the paths in $\paths_{i-1}$ with ends $a_1$ and $b_1$ and ends $a_2$ and $b_2$ respectively, so since $(b_1, \overleftarrow{b_2}) \in \mathcal{F}$ there is a path $Q \in \mathcal{Q}$ with ends $b_1$ and $\overleftarrow{b_2}$. Similarly, at each step $i \in [s] \sm [z]$ we chose some $(a_1, b_1)$ and $(a_2, b_2)$ in $\mathcal{E}_{i-1}$ and added the pair $(b_1, a_2)$ to~$\mathcal{F}_{i-1}$ to form $\mathcal{F}_i$, and we now let $P_1, P_2 \in \paths_{i-1}$ be the paths in $\paths_{i-1}$ with ends $a_1$ and $b_1$ and ends $a_2$ and $b_2$ respectively, so since $(b_1, a_2) \in \mathcal{F}$ there is a path $Q \in \mathcal{Q}$ with ends $b_1$ and $a_2$. In either case we form $\paths_i$ from $\paths_{i-1}$ by removing $P_1$ and $P_2$ and adding $P_1QP_2$, which has ends $a_1$ and $\overleftarrow{a_2}$ in the former case and ends $a_1$ and $b_2$ in the latter case. By induction on $i$ it follows that for each $0 \leq i \leq s-1$ the set $\paths_i$ is a collection of $s-i$ vertex-disjoint $\ell$-paths in $H_p^+$ for which the elements of $\mathcal{E}_i$ are precisely the pairs of ends of members of $\paths_i$ and such that every path in $\paths$ is a path segment of some path in $\paths_i$ (the case $i=0$ is provided by our choice of $\mathcal{E}_0$ and $\paths_0 = \paths$). The case $i=s-1$ then yields that $\paths_{s-1}$ consists of a single $\ell$-path $P^*$ in $H_p^+$ which contains every path in $\paths$ as a path segment and whose ends are $a^*$ and $b^*$, where $(a^*, b^*)$ is the unique element of $\mathcal{E}_{s-1}$. Since $(a^*, b^*)$ must then have been added to $\mathcal{F}_s$ in the final step, it follows that there is an $\ell$-path $Q^* \in \mathcal{Q}$ with ends $b^*$ and $a^*$, and $C = P^*Q^*$ is then an $\ell$-cycle in $H_p^+$ which contains every path in $\paths$ as a path segment. We also have $|V(C) \cap X| \leq b\theta n \leq 4k\theta n$ since the only vertices used from $X$ are the at most $b$ vertices of $X$ contained in each of the at most $\theta n$ paths in $\mathcal{Q}$. 
\end{proof}

\section{Concluding remarks}\label{sec:conclusion}

Our proof of Theorem~\ref{main} used an absorbing argument in which both the connecting and absorbing structures were `composite' structures formed partly of non-random edges in the $k$-graph $H$ and partly of random edges added to form $H_p^+$. We are not aware of any previous uses of this approach, but we believe it may prove useful for a range of related problems.

It is not too difficult to modify the proof of Theorem~\ref{main} to give an alternative proof of Theorem~\ref{perturbed1cycles}. Indeed, given integers $k$ and $m$, for $p = O(n^{-(k-1)})$ it is straightforward to show that the random $k$-graph $\hnpk$ admits an almost-spanning collection of vertex-disjoint $1$-paths of length~$m$. Moreover, it is not too difficult to adapt the proofs of Lemmas~\ref{Lem:Absorbing_Paths} and~\ref{Lemma:ConnectingPaths} to show that each of these statements holds for $1$-paths also in this probability regime. Finally slight changes are needed to the proof of Theorem~\ref{main} to reflect that the path-tiling is almost-spanning rather than spanning, but this presents no great difficulty.

While we were finalising this paper two  further advances on large structures in randomly-perturbed graphs were publicised: Balogh, Treglown and Wagner~\cite{BTW} proved an analogue of Theorem~\ref{perturbedcyclesgraphs} for perfect $H$-tilings, whilst B\"ottcher, Montgomery, Parczyk and Person~\cite{BMPP} gave a similar result for any spanning graph of bounded maximum degree. 

It is natural to ask whether Theorems~\ref{perturbed1cycles} and~\ref{main} remain valid if we replace the minimum $\ell'$-degree condition by a weaker type of minimum degree condition. In particular, do analogous results hold if we instead assume the minimum vertex degree condition $\delta_1(H) \geq \alpha n^{k-1}$?

\providecommand{\bysame}{\leavevmode\hbox to3em{\hrulefill}\thinspace}

\end{document}